\documentclass{article}
\usepackage[T1]{fontenc}
\usepackage{titlesec}
\usepackage[utf8]{inputenc}
\usepackage{amsmath}
\usepackage{amssymb}
\usepackage{amsfonts}
\usepackage{amsthm}
\usepackage{geometry}
\usepackage{xcolor}
\usepackage[hidelinks]{hyperref}
\newtheorem{lemma}{Lemma}
\newtheorem{theorem}{Theorem}

\title{PhD}
\author{}
\date{January 2021}

\begin{document}

\begin{center}
    \textbf{More Limiting Distributions for Eigenvalues of Wigner Matrices}
\end{center}

\begin{center}
     Simona Diaconu\footnote{Department of Mathematics, Stanford University, sdiaconu@stanford.edu}
\end{center}

\begin{abstract}
   The Tracy-Widom distributions are among the most famous laws in probability theory, partly due to their connection with Wigner matrices. In particular, for \(A=\frac{1}{\sqrt{n}}(a_{ij})_{1 \leq i,j \leq n} \in \mathbb{R}^{n \times n}\) symmetric with \((a_{ij})_{1 \leq i \leq j \leq n}\) i.i.d. standard normal, the fluctuations of its largest eigenvalue \(\lambda_1(A)\) are asymptotically described by a real-valued Tracy-Widom distribution \(TW_1:\) \(n^{2/3}(\lambda_1(A)-2) \Rightarrow TW_1.\) As it often happens, Gaussianity can be relaxed, and this results holds when \(\mathbb{E}[a_{11}]=0, \mathbb{E}[a^2_{11}]=1,\) and the tail of \(a_{11}\) decays sufficiently fast: \(\lim_{x \to \infty}{x^4\mathbb{P}(|a_{11}|>x)}=0,\) whereas when the law of \(a_{11}\) is regularly varying with index \(\alpha \in (0,4),\) \(c_a(n)n^{1/2-2/\alpha}\lambda_1(A)\) converges to a Fréchet distribution for \(c_a:(0,\infty) \to (0,\infty)\) slowly varying and depending solely on the law of \(a_{11}.\) This paper considers a family of edge cases, \(\lim_{x \to \infty}{x^4\mathbb{P}(|a_{11}|>x)}=c \in (0,\infty),\) and unveils a new type of limiting behavior for \(\lambda_1(A):\) a continuous function of a Fréchet distribution in which \(2,\) the almost sure limit of \(\lambda_1(A)\) in the light-tailed case, plays a pivotal role:     
   \[f(x)=\begin{cases}
    2, & 0<x<1 \\
    x+\frac{1}{x}, & x \geq 1 
    \end{cases}.\]  
\end{abstract}

\section{Introduction}\label{intro}

Wigner matrices have been an object of intensive study in mathematics ever since Eugene Wigner proposed them in \(1955\) as a tool for understanding the organization of heavy nuclei and showed their empirical spectral distribution converges to the semicircle law ([\ref{wigner}]). Such matrices are generally square with entries in \(\mathbb{R}, \mathbb{C},\) or \(\mathbb{H},\) and satisfy certain symmetry conditions: the focus hereafter is the real-valued symmetric case. Let \(A=\frac{1}{\sqrt{n}}(a_{ij})_{1 \leq i,j \leq n} \in \mathbb{R}^{n \times n}\) be a symmetric matrix with i.i.d. entries on its upper triangular component for which \(\mathbb{E}[a_{11}]=0, \mathbb{E}[a^2_{11}]=1,\) and denote by \(\lambda_1(A) \geq \lambda_2(A) \geq ... \geq \lambda_n(A)\) its eigenvalues. When \(a_{11}\) is Gaussian, it is well-known that the edge (i.e., a fixed number of the largest or smallest) eigenvalues of \(A\) exhibit fluctuations described by Tracy-Widom distributions (see, for instance, the seminal paper [\ref{tracywidom}]) and several universality results, meant to deal with the case in which \(a_{11}\) is not normally distributed, have been discovered (e.g., Tao and Vu~[\ref{taovu}]). 
\par
Consequently, a natural question is what can substitute the Gaussianity assumption in such results. It must be mentioned that this condition cannot be completely dispensed with: as a finite second moment of the entries is necessary for the convergence of the empirical spectral distribution of \(A\) to the semicircle law, the fourth moment is crucial for the asymptotic behavior of \(\lambda_1(A)\) (Bai and Yin~[\ref{baiyin}] showed a finite fourth moment is required if the largest eigenvalue has an almost sure deterministic limit; sample covariance matrices with the number of samples proportional to their dimension represent another instantiation of this phenomenon: when the fourth moment is finite, the largest eigenvalue tends almost surely to a constant, whereas when the former is infinite, the latter tends to infinity with probability one: see Bai and Yin~[\ref{baiyin2}], Bai et al.~[\ref{baietal}]). Furthermore, if \(a_{11}\) is heavy-tailed (i.e., its law is regularly varying of index \(\alpha \in (0,4):\) at a high level, this says \(\mathbb{P}(|a_{11}|>x)\) decays like \(x^{-\alpha},\) and in particular, its \(\alpha\)-moment is infinite), then a new behavior emerges: the edge eigenvalues, properly normalized, fluctuate according to a Poisson point process (Soshnikov~[\ref{soshnikov}] studied \(\alpha \in (0,2),\) and Auffinger et al.~[\ref{auffinger}] extended this result to \(\alpha \in (0,4)\)). 
\par
The question of finding optimal conditions under which the edge eigenvalues can be described by a Tracy-Widom distribution received a fair amount of attention and was completely answered in a paper of Lee and Yin~[\ref{leeyin}]: this occurs if and only if \(\lim_{x \to \infty}{x^4\mathbb{P}(|a_{11}|\geq x)}=0.\) It must be noticed there had been several publications prior to this result, proving an \(\alpha\)-finite moment of the underlying distribution suffices (for symmetric distributions, Ruzmaikina~[\ref{ruzmaikina}] obtained \(\alpha>18,\) and later this was improved to \(\alpha>12\) by Khorunzhiy~[\ref{khorunzhiy}]). This paper is concerned with a family of edge cases, distributions for which \(\lim_{x \to \infty}{x^4\mathbb{P}(|a_{11}|\geq x)}=c \in (0,\infty)\) and the main result is: 

\begin{theorem}\label{theorem1}
    Suppose \(A=\frac{1}{\sqrt{n}}(a_{ij})_{1 \leq i,j \leq n} \in \mathbb{R}^{n \times n}\) is a symmetric matrix for which \((a_{ij})_{1 \leq i \leq j \leq n}\) are i.i.d. and the distribution of \(a_{11}\) is symmetric with \(\mathbb{E}[a^2_{11}]=1, \lim_{x \to \infty}{x^4\mathbb{P}(|a_{11}|\geq x)}=c \in (0,\infty).\) 
    Then as \(n \to \infty,\)
    \begin{equation}\label{conv}
        \lambda_1(A) \Rightarrow f(\zeta_c),
    \end{equation}
    where 
    \[f(x)=\begin{cases}
    2, & 0<x<1 \\
    x+\frac{1}{x}, & x \geq 1 
    \end{cases},\]
    and \(\zeta_c>0\) has a Fréchet distribution with shape and scale parameters \(4,(\frac{c}{2})^{1/4},\) respectively: for all \(x>0,\)
    \(\mathbb{P}(\zeta_c \leq x)=\exp(-\frac{cx^{-4}}{2}).\) 
\end{theorem}

\par
Several observations are in order regarding the convergence stated in (\ref{conv}). First, the limiting distribution arises at the collision of heavy- and light-tailed regimes. More precisely, it inherits the Fréchet fluctuations \(\zeta_c\) from the Poisson point process characterizing the extrema of heavy-tailed i.i.d. random variables, whereas \(2\) is a vestige from light-tailed distributions since the convergence of the empirical spectral distribution of \(A\) to the semicircle law holds as long as \(a_{11}\) has its second moment finite (subsection~\ref{seclowerbound} expounds on this phenomenon). 
\par
Second, the function \(f\) is tightly related to a sequence of polynomials \((s(d,X))_{d \in \mathbb{N}},\) whose coefficients are non-negative and have a combinatorial description: specifically, \(s(d,X)\) has degree \(2d-2,\) and for \(x>0,\)
\[f(x)=\lim_{d \to \infty}{s(d,x)^{1/2d}}\]
(conditionally on an event \(E(x),\) a trace will be roughly \(\lambda^{2d}_1(\frac{1}{\sqrt{n}}A)\) and of order \(s(d,x)\) by counting).
The strategy adopted here has been oftentimes employed for getting a hold of the largest eigenvalue of a symmetric random matrix \(M:\) controlling by careful counting
\[\mathbb{E}[tr(M^{p})]=\mathbb{E}[\sum_{1 \leq i \leq n}{\lambda_i^p(M)}]\]
for large integers \(p\) (e.g., Bai and Yin~[\ref{baiyin}], Benaych-Georges and Péché~[\ref{benaychpeche}], Auffinger et al.~[\ref{auffinger}]), an approach whose by-product is the above definition of \(f.\) Nevertheless, in the current situation, the classical choice \(M=(a_{ij}\chi_{|a_{ij}| \leq c(n)}),\) for a suitable \(c(n),\) falls short due to the heavy tail of \(a_{11}.\) To illustrate how this occurs, suppose the goal is bounding \(||A||,\) and let \(p\) be even so that \(||M||^p \leq tr(M^{p}).\) After truncating \(A\) (to ensure all moments are finite), two incompatible constraints on \(p\) emerge: it must be large to annihilate the contribution of the other eigenvalues (since the empirical spectral distribution of \(A\) converges to the semicircle law, such trace would be at least of order \(n \cdot 2^{p}:\) thus, to eliminate \(n,\) \(p\) should grow faster than \(\log{n}\)), but also small to deal with the non-negligible terms, which are numerous because the moments of \(a_{11}\) grow fast. 
\par
This failure suggests the necessity of twisting this method to adapt it to the present context: an ideal substitute of \(tr(M^{p})\) would be on the one hand, lighter than what is meant to replace, and on the other hand, amenable to combinatorics. In light of these observations, a promising candidate is 
\[tr((S+Q)^p)-tr(Q^p)\] 
where \(S,Q\) are symmetric with \((C1) \hspace{0.1cm} ||S+Q-A||\) small, and \((C2) \hspace{0.1cm} S\) very sparse (say, \(O(n)\) non-zero entries) inasmuch as \(C1\) would allow switching from \(A\) to \(S+Q,\) while \(C2\) would ensure a considerable overlap between the eigenvalues of \(S+Q\) and those of \(Q,\) the difference above hence generating plenty of cancellations (see Lemma~\ref{linalglemma} for a rigorous statement). The desired convergence concerning \(A\) is thus justified by constructing such a proxy \(S+Q,\) further analyzed with the aid of the counting technique developed by Sinai and Soshnikov in [\ref{sinaisosh}]. Some modifications are anew indispensable: although both situations share the family of cycles dominating the considered expectations, in the current setting, there exist several types of comparable contributions (these underlie the sequence of polynomials \(s(d,X)\) mentioned earlier), whereas in the framework of [\ref{sinaisosh}], each dominating cycle generates the same value. Furthermore, the expectation in this case is not unconditional (a conditioning is employed to freeze the largest entries of \(A\)).    
\par
Third, for \(k\) fixed, the joint distribution of the \(k\) largest eigenvalues of \(A\) can be determined reasoning as in Soshnikov~[\ref{soshnikov}]. Theorem \(1.2\) of [\ref{soshnikov}] states that for regularly varying distributions \(a_{11}\) with index \(\alpha \in (0,2),\) the limiting law of the positive eigenvalues of \(A\) (appropriately normalized) is given by an inhomogeneous Poisson point process \(N\) on \((0,\infty)\) with intensity \(\rho(x)=\frac{\alpha}{x^{\alpha+1}}:\) the ingredients behind this result are the behavior of \(\lambda_1(A),\) the Cauchy interlacing inequalities, and the theory on extrema of random variables in the domain of attraction of \(\alpha\)-laws (see, for instance, Theorem \(2.3.1\) in Leadbetter et al.~[\ref{leadbetteretal}]). 
In the present situation, the intensity is \(\rho(x)=\frac{c}{8x^{5}},\) and the convergence of the positive eigenvalues of \(A\) is not to the point process itself, but rather to \(f(N)\) (see end of subsection~\ref{largesteigenvalue}). Clearly, by symmetry, results similar to (\ref{conv}) hold for the smallest eigenvalues of \(A.\)
\par
The remainder of the paper contains the proof of Theorem~\ref{theorem1}: subsections \ref{seclowerbound} and \ref{secuppbound} present the rival forces behind the object of interest and the matrix decomposition leading to a proxy as described previously; section~\ref{section2} gathers the necessary tools for showing
\begin{equation}\label{desiredlim}
    \limsup_{n \to \infty}{\mathbb{P}(||A|| \geq x)} \leq \mathbb{P}(f(\zeta_c) \geq x) \leq \liminf_{n \to \infty}{\mathbb{P}(||A|| \geq x)}
\end{equation}
for all \(x \in \mathbb{R};\) section~\ref{section3} consists of proving (\ref{desiredlim}) and justifying why an analogous chain of inequalities holds when \(||A||\) is replaced by \(\lambda_1(A).\)

\subsection{A Lower Bound}\label{seclowerbound}

This subsection presents a preliminary inequality \(\lambda_1(A)\) satisfies: for \(t>0,\)
\begin{equation}\label{liminf}
    \lambda_1(A) \geq \max{(\max{A},2)}-t
\end{equation}
with probability tending to one as \(n\) tends to infinity, where \(\max{A}:=\frac{1}{\sqrt{n}}\max_{1 \leq i \leq j \leq n}{|a_{ij}|}.\) 
Although this result is not directly employed to prove Theorem~\ref{theorem1}, it displays the two essential quantities underlying both the operator norm of \(A\) and its largest eigenvalue. Henceforth, an event \(E=E_n\) is said to hold with high probability if \(\lim_{n \to \infty}{\mathbb{P}(E_n)}=1.\)
\par
On the one hand, since \(A\) is a normalized Wigner matrix whose entries have variance \(\frac{1}{n},\) its empirical spectral distribution converges almost surely to the semicircle law \(\rho\) (theorem \(2.5\) in Bai and Silverstein~[\ref{baisilv}]). Thus, almost surely
\[\liminf_{n \to \infty}{\lambda_{1}(A)} \geq 2\]
as \(\rho\) assigns a positive mass to any neighborhood of \(2,\) from which for \(t>0,\) with high probability
\begin{equation}\label{firsteasy}
   \lambda_1(A) \geq 2-t.
\end{equation}
\par
On the other hand,
\begin{equation}\label{maxdist}
   \max{A}=\frac{1}{\sqrt{n}}\max_{1 \leq i \leq j \leq n}{|a_{ij}|} \Rightarrow \zeta_c
\end{equation}
because for \(s>0\) as \(n \to \infty,\)
\[\mathbb{P}(\frac{1}{\sqrt{n}}\max_{1 \leq i \leq j \leq n}{|a_{ij}|}<s)=(1-\mathbb{P}(|a_{11}| \geq s\sqrt{n}))^{\frac{n^2+n}{2}}=\exp(-c(s\sqrt{n})^{-4}(1+o(1)) \cdot \frac{n^2+n}{2}) \to \exp(-\frac{c}{2s^4}).\] 
Let \(|a_{i_0j_0}|=\max_{1 \leq i \leq j \leq n}{|a_{ij}|}\)
with \(i_0 \leq j_0:\) (\ref{maxdist}) and
\begin{equation}\label{maxdiag}
    \frac{1}{\sqrt{n}}\max_{1 \leq i \leq n}{|a_{ii}|} \xrightarrow[]{p} 0   
\end{equation}
((\ref{maxdist}) implies \(\frac{1}{n^{1/4}}\max_{1 \leq i \leq n}{|a_{ii}|} \Rightarrow \zeta_{2c}\)) yield \(i_0<j_0\) with high probability. Consider a unit vector \(v \in \mathbb{R}^{n}\) with \(|v_{i_0}|=|v_{j_0}|=\frac{1}{\sqrt{2}}\) and \(v_{i_0}v_{j_0}a_{i_0j_0} \geq 0.\) Therefore,
\[v^TAv=\frac{1}{\sqrt{n}}(|a_{i_0j_0}|+\frac{a_{i_0i_0}}{2}+\frac{a_{j_0j_0}}{2}) \geq \frac{1}{\sqrt{n}}\max_{1 \leq i \leq j \leq n}{|a_{ij}|}-\frac{1}{\sqrt{n}}\max_{1 \leq i \leq n}{|a_{ii}|},\]
which in conjunction with (\ref{maxdiag}) gives for \(t>0,\)
\begin{equation}\label{secondeasy}
    \lambda_{1}(A) \geq \max{A}-t.
\end{equation}
with high probability. The desired bound (\ref{liminf}) ensues from (\ref{firsteasy}) and (\ref{secondeasy}).

\subsection{A Matrix Decomposition}\label{secuppbound}

In light of (\ref{maxdist}), showing for \(M, \epsilon>0,\)
\begin{equation}\label{goal}
    \lim_{n \to \infty}{\mathbb{P}(||A||>f(M)+2\epsilon \hspace{0.05cm} | \hspace{0.05cm} \max{A} \leq M)}=0,
\end{equation}
\begin{equation}\label{goal2}
    \lim_{n \to \infty}{\mathbb{P}(||A||<f(\max{A})-2\epsilon)}=0
\end{equation}
suffices to justify (\ref{desiredlim}).
\par
One core ingredient for both (\ref{goal}) and (\ref{goal2}) is a decomposition of \(A\) into three matrices \(A_s, A_m, A_b,\) with small, medium, and big entries, respectively. The last component, already sparse, is further split into two matrices, one of them being considerably sparser than \(A_b.\) Next, it is proved that with high probability \(||A_m||\) is negligible, while the first component of \(A_b\) contributes at most \(\epsilon.\) Thus, in an operator norm sense, the sum of \(A_s\) and the sparser component of \(A_b\) differs from \(A\) by at most \(\epsilon,\) making the former a proxy for the latter.
\par
Let \(\delta_1, \delta_3 \in (0,\frac{1}{64}), \delta_2 \in (0,\frac{1}{32})\) be fixed constants and
\[A_s=\frac{1}{\sqrt{n}}(a_{ij}\chi_{|a_{ij}| \leq n^{1/4-\delta_1}}), \hspace{0.1cm} A_m=\frac{1}{\sqrt{n}}(a_{ij}\chi_{n^{1/4-\delta_1}<|a_{ij}| \leq n^{3/8+\delta_2}}), \hspace{0.1cm} A_b=\frac{1}{\sqrt{n}}(a_{ij}\chi_{|a_{ij}|> n^{3/8+\delta_2}}),\]
for which
\[A=A_s+A_m+A_b.\]
Theorem \(2.1\) of Benaych-Georges and Péché~[\ref{benaychpeche}], stated below, is employed next to bound \(||A_m||\) (as well as \(||A_s||\) in later subsections).

\begin{theorem}[{Benaych-Georges and Péché~[\ref{benaychpeche}]}]\label{thbenaychpeche}
     Suppose \(\tilde{A}=(a_{ij})_{1 \leq i,j \leq n}\) is a symmetric real-valued random matrix with at most \(n^{\mu}\) non-zero entries on each row, \((a_{ij})_{1 \leq i \leq j \leq n}\) i.i.d., of variance one, with distribution symmetric and regularly varying of index \(\alpha>2.\) Then for \(A_n=(a_{ij}\chi_{|a_{ij}| \leq n^{\gamma}})\) and any constants \(\gamma, \gamma', \gamma''>0\) with \(\frac{\mu}{2} \leq \gamma', \frac{\mu}{4}+\gamma+\gamma''<\gamma',\)
    \[\mathbb{E}[tr(A_n^{2s_n})] \leq L(n)n^{1+2\gamma}s_n^{-3/2}(2n^{\gamma'})^{2s_n}\]
    for a slowly varying function \(L,\) and all \(s_n \in \mathbb{N}, s_n \leq n^{\gamma''}.\)
\end{theorem}

\par
Notice Theorem~\ref{thbenaychpeche} holds for matrices of the form \(A_n=(a_{ij}\chi_{n^{\gamma_0}<|a_{ij}| \leq n^{\gamma}})\) too since a lower bound on the entries can only decrease the moments of the random variables appearing in the trace and the number of non-zero elements among them. Consider \(\tilde{A}_m:=\sqrt{n}A_m=(a_{ij}\chi_{n^{1/4-\delta_1}<|a_{ij}| \leq n^{3/8+\delta_2}}).\) The probability of \(\tilde{A}_m\) having at least \(l\) non-zero entries on a given row is at most
\[\binom{n}{l}(2cn^{1/4-\delta_1})^{-4l} \leq \frac{n^l}{l!} \cdot (2c)^{-4l}n^{-l+4\delta_1l}=(2c)^{-4l} \cdot \frac{n^{4\delta_1l}}{l!} \leq (2c)^{-4l} \cdot \frac{n^{4\delta_1l}}{(l/e)^l} \leq (2c/e)^{-4l}n^{-4\delta_1 l}\]
for \(l=n^{8\delta_1}, n \geq n(c).\) Hence for \(E_0,\) the event that each row of \(\tilde{A}_m\) has at most \(l\) non-zero entries,
\[\mathbb{P}(E_0^c) \leq n(2c/e)^{-4l}n^{-4\delta_1 l}=o(1).\]
Conditioning on \(E_0,\) Theorem~\ref{thbenaychpeche} yields
\[\mathbb{E}[tr(\tilde{A}_m^{2s_n}) \hspace{0.05cm} | \hspace{0.05cm}  E_0] \leq L(n)n^{1+2\gamma}s_n^{-3/2}(2n^{\gamma'})^{2s_n}\]
for 
\[\mu=8\delta_1, \hspace{0.1cm} \gamma=\frac{3}{8}+\delta_2, \hspace{0.05cm} \gamma''=\delta_3, \hspace{0.1cm} \gamma'=\frac{\mu}{4}+\gamma+2\gamma''=2\delta_1+\frac{3}{8}+\delta_2+2\delta_3 \in (\frac{\mu}{2},\frac{1}{2})=(4\delta_1,\frac{1}{2})\]
(subsection~\ref{twistedmethod} presents in detail why such conditional expectations can replace their unconditional counterparts at the cost of a factor \(c(\epsilon)\) for \(\epsilon>0,\) which can be evidently absorbed by \(L\)). Chebyshev's inequality then gives for \(\delta_4>0\) and \(n\) sufficiently large,
\[\mathbb{P}(tr(\tilde{A}_m^{2s_n}) \geq (2n^{\gamma'})^{2s_n}n^{2s_n\delta_4}) \leq \mathbb{P}(E^c_0)+n^{2+2\gamma}n^{-2s_n\delta_4}=o(1)\]
by choosing \(s_n=\lfloor{} n^{\gamma''} \rfloor.\)
Since \(\gamma'<\frac{1}{2},\) \(\delta_4=\frac{1}{2}(\frac{1}{2}-\gamma')>0\) entails
\begin{equation}\label{mediummatrix}
    ||A_m|| \leq n^{-1/2}(tr(\tilde{A}_m^{2s_n}))^{\frac{1}{2s_n}} \leq 2n^{\gamma'+\delta_4-\frac{1}{2}}=o(1)
\end{equation}
with high probability.
\par
Proceed now with the split of \(A_b.\) Let \(E_1\) be the event that the non-zero entries of \(A_b\) are off-diagonal and any two lie on different rows: by a union bound,
\[\mathbb{P}(E_1^c) \leq n \cdot 2c(n^{3/8+\delta_2})^{-4} + n^2 \cdot n \cdot (2c(n^{3/8+\delta_2})^{-4})^2=o(1)\]
for \(n \geq n(c).\) For \(\kappa>0\) and a sequence \(m=m_n \to \infty,\) let
\[A_b=\frac{1}{\sqrt{n}}(a_{ij}\chi_{n^{3/8+\delta_2} \leq |a_{ij}| \leq \kappa \sqrt{n}})+\frac{1}{\sqrt{n}}(a_{ij}\chi_{|a_{ij}|> \kappa \sqrt{n}}):=A_{b,\kappa}+A_{B,\kappa},\]
and \(E_2\) the event that \(A_{B,\kappa}\) has at most \(2m\) non-zero entries. Then
\[\mathbb{P}(E_2^c) \leq \binom{n^2}{m}\cdot (2c(\kappa \sqrt{n})^{-4})^m \leq \frac{n^{2m}}{m!} \cdot (2c\kappa^{-4})^mn^{-2m}=\frac{(2c\kappa^{-4})^m}{m!}=o(1)\]
(at least \(m\) elements of size at least \(\kappa \sqrt{n}\) must exist among the \(\frac{n^2+n}{2} \leq n^2\) i.i.d. random variables \((a_{ij})_{1 \leq i \leq j \leq n}\)). Moreover, when \(E_1\) occurs, \(A_{b,\kappa}\) has at most one non-zero entry per row and so
\begin{equation}\label{bigish}
    ||A_{b,\kappa}|| \leq \max_{1 \leq i,j \leq n}{|(A_{b,\kappa})_{ij}|} \leq \kappa.
\end{equation}
\par
In virtue of (\ref{mediummatrix}) and (\ref{bigish}), for any \(\kappa \leq \epsilon\) and fixed sequence \(m=m_n \to \infty,\) (\ref{goal}) and (\ref{goal2}) ensue from
\begin{equation}\label{goalkappa}\tag{8'}
    \lim_{n \to \infty}{\mathbb{P}(||A_\kappa||>f(M)+\epsilon \hspace{0.05cm} | \hspace{0.05cm} E_2, \max{A} \leq M)}=0,
\end{equation}
\begin{equation}\label{goalkappa2}\tag{9'}
    \lim_{n \to \infty}{\mathbb{P}(||A_\kappa||<f(\max{A})-\epsilon| \hspace{0.05cm} E_2})=0,
\end{equation}
where \(A_\kappa:=A_s+A_{B,\kappa}.\) These two limits are the subject of the forthcoming section.

\section{Conditional Operator Norms}\label{section2}

Identities ((\ref{goalkappa}) and (\ref{goalkappa2})) are justified by analyzing 
\[tr(A_{\kappa}^{2p})-tr(A_s^{2p})=tr((A_s+A_{B,\kappa})^{2p})-tr(A_s^{2p})\] 
for large integers \(p.\) Roughly speaking, Weyl's inequalities and the sparsity of \(A_{B,\kappa}\) entail this difference grows at the same rate as \(||A_{\kappa}||^{2p}.\) Furthermore, by conditioning on the appropriate events, its expectation can be squeezed between \(s(p,\max{A})\) and \(s(p,M),\) (up to constant powers \(p\)), where \(s: \mathbb{N} \times (0,\infty) \to [0,\infty)\) is the polynomial function yielding the corresponding conditional expectations, whenever \(m=m_n,p=p_n,n\) grow to infinity at completely different rates (\(m \leq \log{\log{p}}, p \leq \log{\log{n}}\) suffice). Henceforth such a growth hierarchy is implicitly assumed.
\par
Let us introduce the notation needed for the conditionings to come. Denote by \(\mathcal{S}=\mathcal{S}_{n,m}\) the set of subsets \(S \subset \{(i,j): 1 \leq i \leq j \leq n\}\) with the following properties: 
\par
\((a) |S| \leq m;\) 
\par
\((b)\) any \((i,j) \in S\) has \(i<j;\)
\par
\((c)\) all pairwise distinct elements \((i_1,j_1),(i_2,j_2)\) of \(S\) satisfy \(\{i_1,j_1\} \cap \{i_2,j_2\}=\emptyset. \newline\) 
Consider the events
\[E(S,\kappa,M)=\{\max_{i \leq j, (i,j) \not \in S}{|a_{ij}|} \leq \kappa \sqrt{n}<\min_{(i,j) \in S}{|a_{ij}|} \leq \max_{(i,j) \in S}{|a_{ij}|} \leq M\sqrt{n}\},\]
\[E^{B}(S,\kappa,M)=E(S,\kappa,M) \cap \{a_{ij}, (i,j) \in S\}\]
(the first requires the set of positions of the non-zero entries of \(A_{B,\kappa}\) to be \(S \cup \{(j,i): (i,j) \in S\},\) while the second also fixes their values). Clearly, \((E(S,\kappa,M))_{S \in \mathcal{S}}\) are pairwise disjoint, and
\[E_1 \cap E_2 \cap E_3 \subset \cup_{S \in \mathcal{S}}{E(S,\kappa,M)}\]
where
\[E_3=\{\max_{1 \leq i,j \leq n}{|a_{ij}|} \leq M\sqrt{n}\}.\] 
Since \(E_1 \cap E_2\) has probability tending to one as \(n \to \infty,\) a sufficient condition for (\ref{goalkappa}) is
\begin{equation}\label{firstfinalgoal}\tag{8''}
    \lim_{n \to \infty}{\mathbb{P}(||A_\kappa||>f(M)+\epsilon \hspace{0.05cm} | \hspace{0.05cm} E(S,\kappa,M))}=0
\end{equation}
uniformly in \(S \in \mathcal{S}\) (i.e., the bounds involve solely \(M, m, n\)). Inequality (\ref{goalkappa2}) ensues from a similar uniform convergence in \(S \in \mathcal{S}, S \ne \emptyset:\)
\begin{equation}\label{secondfinalgoal}\tag{9''}
    \lim_{n \to \infty}{\mathbb{P}(||A_\kappa||<f(\max{A})-\epsilon \hspace{0.05cm} | \hspace{0.05cm} E^B(S,\kappa,M))}=0
\end{equation}
(this yields
\[\limsup_{n \to \infty}{\mathbb{P}(||A_\kappa||<f(\max{A})-\epsilon)} \leq \mathbb{P}(\max{A} \not \in [\kappa,M]);\]
then use (\ref{maxdist}), and let \(\kappa \to 0,M \to \infty\)). For the sake of simplicity, denote conditioning on these events by \(*,**,\) respectively (i.e., by an abuse of notation,
\[\mathbb{P}_*(||A_\kappa||>f(M)+\epsilon):=\mathbb{P}(||A_\kappa||>f(M)+\epsilon \hspace{0.05cm} | \hspace{0.05cm} E(S,\kappa,M)),\]
\[\mathbb{P}_{**}(||A_\kappa||<f(\max{A})-\epsilon):=\mathbb{P}(||A_\kappa||<f(\max{A})-\epsilon \hspace{0.05cm} | \hspace{0.05cm} E^B(S,\kappa,M)),\]
for \(S \in \mathcal{S}\) fixed).
\par
The means of bounding the conditional probabilities in (\ref{firstfinalgoal}) and (\ref{secondfinalgoal}) is computing 
\[\mathbb{E}_*[tr(A_{\kappa}^{2p})-tr(A_s^{2p})], \hspace{0.2cm} \mathbb{E}_{**}[tr(A_{\kappa}^{2p})-tr(A_s^{2p})],\]
by employing the combinatorial technique behind the proof of Theorem~\ref{thbenaychpeche}, pioneered by Sinai and Soshnikov in [\ref{sinaisosh}], and subsequently used in several contexts (e.g., Sinai and Soshnikov~[\ref{sinaisosh2}], Soshnikov~[\ref{sohnikov2}], Auffinger et al.~[\ref{auffinger}]).
\par
In the rest of this section,
\begin{itemize}
    \item \ref{method} presents in detail the combinatorial method of Sinai and Soshnikov in [\ref{sinaisosh}];
    
    \item \ref{twistedmethod} uses this counting device on the first conditional expectation above;
    
    \item \ref{closedformcombfunction} connects the functions \(s\) and \(f:\) \(\lim_{p \to\infty}{s(M,p)^{1/2p}}=f(M)\) for \(M>0.\)
\end{itemize}

\subsection{Large Moments}\label{method}

Suppose \(p \in \mathbb{N}\) and \(B=(b_{ij})_{1 \leq i,j \leq n} \in \mathbb{R}^{n \times n}\) is a symmetric random matrix for which \((b_{ij})_{1 \leq i \leq j \leq n}\) are i.i.d. and \(b_{11}\) has a symmetric distribution with \(\mathbb{E}[b_{11}^2] \leq 1, \mathbb{E}[b_{11}^{2l}] \leq L(n)n^{\delta(2l-4)}, 2 \leq l \leq p, \delta>0,\) and \(L:\mathbb{N} \to [1,\infty), L(n)<n^{2\delta}.\) The content of this subsection is
\[\mathbb{E}[tr(B^{2p})] \leq L(n)2^{2p}p!\sum_{(n_1,\hspace{0.05cm} ...\hspace{0.05cm} ,n_p)}{n^{1+\sum_{1 \leq k \leq p}{n_k}+2\delta\sum_{k \geq 2}{kn_k}}\prod_{1 \leq k \leq p}{\frac{1}{(k!)^{n_k}n_k!}}\prod_{2 \leq k \leq p}{(2k)^{kn_k}}}\]
where \(n_1,\hspace{0.05cm} ...\hspace{0.05cm} ,n_p\) are non-negative integers with \(\sum_{1 \leq k \leq p}{kn_k}=p.\) This inequality and \(||B||^{2p} \leq tr(B^{2p})\) provide some non-trivial information about \(||B||\) as long as \(\delta\) and \(p\) are small enough, in which case the sum on the right-hand side is bounded by a simple expression (these computations are included at the end of this subsection).
\par
Clearly,
\begin{equation}\label{gentrace}
    \mathbb{E}[tr(B^{2p})]=\sum_{(i_0, i_1, \hspace{0.05cm} ... \hspace{0.05cm}, i_{2p-1})}{\mathbb{E}[b_{i_0i_1}b_{i_1i_2}...b_{i_{2p-1}i_{0}}]}.
\end{equation}
Let \(\mathbf{i}:=(i_0,i_1, \hspace{0.05cm} ...\hspace{0.05cm} ,i_{2p-1},i_0)\) and \(b_{\mathbf{i}}:=b_{i_0i_1}b_{i_1i_2}...b_{i_{2p-1}i_{0}}.\) Interpret \(\mathbf{i}\) as a directed cycle with vertices among \(\{1,2, \hspace{0.05cm} ... \hspace{0.05cm}, n\}\) and call \((i_{k-1},i_k)\) its \(k^{th}\) edge for \(1 \leq k \leq 2p,\) where \(i_{2p}:=i_0;\) for \(u,v \in \{1,2, \hspace{0.05cm} ... \hspace{0.05cm}, n\},\) \((u,v)\) is a directed edge from \(u\) to \(v,\) whereas \(uv\) is undirected (the former are the building blocks of the cycles underlying the trace in (\ref{gentrace}), while the latter determine their expectations): in particular, \(uv=vu.\) Call \(\mathbf{i}\) an \textit{even cycle} if each undirected edge appears an even number of times in it; using symmetry, \(\mathbb{E}[b_{\mathbf{i}}]=0\) unless \(\mathbf{i}\) is an even cycle.
\par
The crux of the technique developed by Sinai and Soshnikov in~[\ref{sinaisosh}] is a change of summation in (\ref{gentrace}), from even cycles \(\mathbf{i}\) to \(p\)-tuples of non-negative integers \((n_1,n_2, \hspace{0.05cm} ... \hspace{0.05cm} , n_p)\) satisfying \(\sum_{1 \leq k \leq p}{kn_k}=p.\) This is achieved by mapping each such cycle to a tuple of this type, and bounding from above the sizes of the preimages of this transformation and the expectations of their elements. For \(\mathbf{i},\) call an edge \((i_k,i_{k+1})\) and its right endpoint \(i_{k+1}\) \textit{marked} if an even number of copies of \(i_ki_{k+1}\) precedes it: i.e., if \(\{t \in \mathbb{Z}: 0 \leq t \leq k-1, i_ti_{t+1}=i_ki_{k+1}\}\) has even size, and pair each unmarked edge with its last marked copy (i.e., for \((i_k,i_{k+1})\) unmarked, pair it with \((i_{t'},i_{t'+1}),\) where \(t'=\max{\{t \in \mathbb{Z}: 0 \leq t \leq k-1, i_ti_{t+1}=i_ki_{k+1}\}}\)). As it will soon become apparent, the analysis of such cycles, and consequently of the trace, relies on this pairing. Each even cycle \(\mathbf{i}\) has \(p\) marked edges, and any vertex \(j \in \{1,2, \hspace{0.05cm} ... \hspace{0.05cm}, n\}\) of \(\mathbf{i},\) apart perhaps from \(i_0,\) is marked at least once (the first edge of \(\mathbf{i}\) containing \(j\) is of the form \((i,j)\) since \(i_0 \ne j,\) and no earlier edge is adjacent to \(j\)). For \(0 \leq k \leq p,\) denote by \(N_{\mathbf{i}}(k)\) the set of \(j \in \{1,2, \hspace{0.05cm} ... \hspace{0.05cm}, n\}\) marked exactly \(k\) times in \(\mathbf{i}\) with \(n_k:=|N_{\mathbf{i}}(k)|.\) Then
\begin{equation}\label{tuplecond}
    \sum_{0 \leq k \leq p}{n_k}=n, \hspace{0.2cm} \sum_{1 \leq k \leq p}{kn_k}=p.
\end{equation}
\par
Having constructed a \(p\)-tuple of non-negative integers \((n_1,n_2, \hspace{0.05cm} ... \hspace{0.05cm}, n_p)\) satisfying \(\sum_{1 \leq k \leq p}{kn_k}=p\) from an even cycle \(\mathbf{i},\) the final task is obtaining upper bounds on the number of such cycles mapped to a given tuple (steps \(1-4\)) and their individual contributions (step \(5\)). In what follows, \((n_1,n_2, \hspace{0.05cm} ... \hspace{0.05cm}, n_p)\) remains fixed, and \(\mathbf{i}\) is any even cycle mapped to it by the procedure described above.
\par
\underline{Step \(1.\)} Map \(\mathbf{i}\) to a Dyck path \((s_1,s_2, \hspace{0.05cm} ... \hspace{0.05cm} ,s_{2p}),\) where \(s_k=+1\) if \((i_{k-1},i_k)\) is marked, and \(s_k=-1\) if \((i_{k-1},i_k)\) is unmarked. The number of such paths is the Catalan number \(C_p=\frac{1}{p+1}\binom{2p}{p}.\)
\par
\underline{Step \(2.\)} Once the positions of the marked edges in \(\mathbf{i}\) are chosen (i.e., a Dyck path), establish the order of their marked vertices. There are at most 
\[\frac{p!}{\prod_{1 \leq k \leq p}{(k!)^{n_k}}} \cdot \frac{1}{\prod_{1 \leq k \leq p}{n_k!}}\]
possibilities as each is a partition of a set of size \(p\) in \(n_1+...+n_p\) subsets with \(n_k\) of them of size \(k.\)
\par
\underline{Step \(3.\)} Select the distinct vertices appearing in \(\mathbf{i},\)
\[V(\mathbf{i}):=\cup_{0 \leq k \leq 2p-1}{\{i_k\}},\] 
one at a time, by reading the edges of \(\mathbf{i}\) in order, starting at \((i_0,i_1).\) There are at most 
\[n^{1+\sum_{1 \leq k \leq p}{n_k}}\] 
such sets because \(|V(\mathbf{i})| \leq 1+\sum_{1 \leq k \leq p}{n_k}\) (recall that any vertex of \(\mathbf{i},\) except perhaps from \(i_0,\) is marked at least once).
\par
\underline{Step \(4.\)} Choose the remaining vertices of \(\mathbf{i}\) from \(V(\mathbf{i}),\) by reading anew the edges of \(\mathbf{i}\) in order, beginning at \((i_0,i_1)\) (step \(3\) only established the first appearance of each element of \(V(\mathbf{i})\) in \(\mathbf{i}\)). Observe that only the right ends of the unmarked edges have yet to be decided: the first edge \((i_0,i_1)\) is fixed as \(i_0, i_1\) have already been chosen (\(i_1\) is marked); by induction, any subsequent edge has its left end fixed, and therefore only its right end has yet to be chosen. This yields that marked edges are fully labeled: step \(2\) determines their positions in \(\mathbf{i},\) while step \(3\) appoints their right endpoints. 
\par
The number of possibilities in this case is at most \(\prod_{2 \leq k \leq p}{(2k)^{kn_k}}:\)

\begin{lemma}\label{lemma1}
If \(v \in N_{\mathbf{i}}(k),\) then the number of unmarked edges of the form \((v,u)\) is at most 
\(
\begin{cases} 
2k, & k \geq 1\\
1, & k=1\\
\end{cases}.
\)
\end{lemma}

\begin{proof}
Let \((t_j,v)=(i_{n_j},i_{n_j+1}), 1 \leq j \leq k\) for \(n_1<n_2< ... <n_k\) be the marked edges with right endpoints \(v,\) and \(u_j\) the number of unmarked edges of the type \((v,u)\) with index (i.e., position in \(\mathbf{i}\)) in \([n_j,n_{j+1}],\) where \(n_{k+1}:=2p.\) Because there is no edge of the latter type preceding \((t_1,v),\) and \((t_j,v), 1 \leq j \leq k\) are marked, the statement above is equivalent to 
\(\begin{cases} 
\sum_{1 \leq j \leq k}{u_j} \leq 2k, & k \geq 1\\
u_1 \leq 1, & k=1
\end{cases}.\)
\par
For \(1 \leq j \leq k,\) denote by \(S_j\) the set of marked edges adjacent to \(v,\) with index \(n<n_{j},\) and unmarked counterparts of index \(n' \geq n_j.\) Take \(a_j:=|S_j|,\) fix an integer \(j \in [1,k],\) and suppose 
\[(i_{n_j},i_{n_j+1}, \hspace{0.05cm} ... \hspace{0.05cm}, i_{n_{j+1}},i_{n_{j+1}+1})=(t_j,v,s_1,*,\tilde{s}_1,v,s_2, \hspace{0.05cm} ... \hspace{0.05cm} ,\tilde{s}_{l_j},v,s_{l_{j}+1},*,t_{j+1},v)\]
where \(*\) are sequences of vertices that do not contain \(v.\) Then the edges \((\tilde{s}_1,v), \hspace{0.05cm} ... \hspace{0.05cm} ,(\tilde{s}_{l_j},v)\) are unmarked and their marked counterparts are among \((t_j,v),(v,s_1),\hspace{0.05cm} ... \hspace{0.05cm},(v,s_{l_j}),\) and the elements of \(S_j.\) This gives
\[a_{j+1} \leq (a_j+l_j+2)-(u_j+l_j)=a_j+2-u_j\] 
because \(S_{j+1} \subset S_j \cup \{(t_j,v),(v,s_1),\hspace{0.05cm} ... \hspace{0.05cm}, (v,s_{l_j}),(v,s_{l_{j}+1})\},\) and at least \(u_j+l_j\) elements of the latter set are not in the former since among them, \(u_j\) are unmarked and \(l_j\) are the marked counterparts of \((\tilde{s}_1,v), \hspace{0.05cm} ... \hspace{0.05cm}, (\tilde{s}_{l_j},v).\) 
\par
Hence, \(u_j \leq a_j-a_{j+1}+2\) for all \(1 \leq j \leq k,\) from which 
\[\sum_{1 \leq j \leq k}{u_j} \leq 2k+a_1-a_{k+1}.\] 
\par
If \(v \ne i_0,\) then \(a_1=0\) and so \(\sum_{1 \leq j \leq k}{u_j} \leq 2k.\) Else, \(v=i_0, a_1 \leq 1\) (if \(i_0=i_1,\) then \(n_1=0, a_1=0;\) suppose next \(i_0 \ne i_1;\) if \(i_j \ne i_0\) for \(2 \leq j<n_1,\) then \(S_1=\{(i_0,i_1)\}, a_1=1;\) otherwise, consider \(m\) minimal with \(2 \leq m< n_1, i_m=i_0;\) then \(i_{m-1}=i_1\) because otherwise \((i_{m-1},i_m)=(i_{m-1},i_0)\) would be marked; ignore \((i_0,i_1,\hspace{0.05cm} ... \hspace{0.05cm}, i_{m-1}),\) and proceed with an analogous analysis for \((i_m, \hspace{0.05cm} ... \hspace{0.05cm} ,i_{n_1}):\) this clipping does not affect the pairs of marked edges adjacent to \(i_0,\) and the process is iterated finitely many times), and
\[a_{k+1} \leq (a_k+1+l_k)-(u_k+l_k)=a_k+1-u_k,\] 
which again yields \(\sum_{1 \leq j \leq k}{u_j} \leq 2k.\) To justify this last inequality, notice that in this situation,
\[(i_{n_k},i_{n_k+1}, \hspace{0.05cm} ... \hspace{0.05cm}, i_{n_{2p}})=(t_k,v,s_1,*,\tilde{s}_1,v,s_2, \hspace{0.05cm} ... \hspace{0.05cm}, \tilde{s}_{l_k-1},v,s_{l_k},*,\tilde{s}_{l_k},v),\]
the edges \((\tilde{s}_1,v), \hspace{0.05cm} ... \hspace{0.05cm}, (\tilde{s}_{l_k},v)\) are unmarked, and \(S_{k+1} \subset S_k \cup \{(t_k,v),(v,s_1),\hspace{0.05cm} ... \hspace{0.05cm},(v,s_{l_k})\}\) with at least \(u_k+l_k\) elements contained in the latter set but not in the former.
\par
Suppose next \(k=1.\) If \(v \ne i_0,\) then
\[\mathbf{i}=(i_0,*,t_1,v,s_1,*,\tilde{s}_1,v,s_2,*, \hspace{0.05cm} ...\hspace{0.05cm}, \tilde{s}_l,v,s_{l+1},*,i_0).\]
Clearly \((\tilde{s}_1,v), \hspace{0.05cm} ... \hspace{0.05cm}, (\tilde{s}_l,v),(v,s_{l+1})\) are unmarked, which implies the other \(l+1\) edges containing \(v\) are marked. If \(v=i_0,\) then
\[\mathbf{i}=(v,s_0,*,t_1,v,s_1,*,\tilde{s}_1,v,s_2,*,\hspace{0.05cm} ... \hspace{0.05cm}, \tilde{s}_l,v,s_{l+1},*,s_{l+2},v);\]
\((\tilde{s}_1,v), \hspace{0.05cm} ... \hspace{0.05cm}, (\tilde{s}_{l},v),(s_{l+2},v)\) are unmarked, and so there is exactly another unmarked edge containing \(v\) (which is of the form \((v,u)\) because \((t_1,v)\) is marked). 
\end{proof}
\par
\underline{Step \(5.\)} Bound the expectation generated by \(\mathbf{i}.\) For any undirected edge \(e=uv,\) denote by \(2k(e)\) the number of times \(e\) appears in \(\mathbf{i}.\) The assumption on the moments of \(b_{11}\) implies
\[\mathbb{E}[b_{\mathbf{i}}] \leq \prod_{e \in \mathbf{i}, k(e) \geq 2}{L(n)n^{\delta(2k(e)-4)}}=L(n)^{E}n^{2\delta(\sum_{k(e) \geq 2}{k(e)}-2E)}\]
where \(E:=|\{e \in \mathbf{i}, k(e) \geq 2\}|.\) Any edge \(e=uv\) with \(k(e) \geq 2,\) except possibly for \(i_0i_1,\) has either \(u\) or \(v\) in \(N_{\mathbf{i}}(k)\) for some \(k \geq 2:\) \(k(uv) \geq 2\) entails either the desired conclusion or \(u,v \in N_{\mathbf{i}}(1),\) in which case \(k(uv)=2,\) and there are two marked copies of \(uv\) in \(\mathbf{i},\) \((u,v), (v,u);\) suppose without loss of generality they appear in this order; then \((u,v)\) is the first edge of \(\mathbf{i}\) (otherwise, for the vertex \(t\) preceding this apparition of \(u,\) the edge \((t,u)\) is marked). This observation gives
\begin{equation}\label{firstinequal}
    \sum_{k(e) \geq 2}{k(e)} \leq E+1+\sum_{k \geq 2}{kn_k},
\end{equation}
whereby for \(n>n(\delta),\)
\[\mathbb{E}[b_{\mathbf{i}}] \leq L(n)^{E}n^{2\delta(1+\sum_{k \geq 2}{kn_k}-E)} \leq L(n)n^{2\delta \sum_{k \geq 2}{kn_k}}\]
using that when \(E=0,\) the left-hand side is at most \(1,\) and \(L(n)<n^{2\delta}\) for \(E \geq 1.\)
\par
Putting together steps \(1-5\) yields
\[\mathbb{E}[tr(B^{2p})] \leq L(n)C_pp!\sum_{(n_1,\hspace{0.05cm} ...\hspace{0.05cm} ,n_p)}{n^{1+\sum_{1 \leq k \leq p}{n_k}+2\delta\sum_{k \geq 2}{kn_k}}\prod_{1 \leq k \leq p}{\frac{1}{(k!)^{n_k}n_k!}}\prod_{2 \leq k \leq p}{(2k)^{kn_k}}},\]
where the summation is over \(p\)-tuples of non-negative integers \((n_1,n_2, \hspace{0.05cm} ... \hspace{0.05cm}, n_p)\) with \(\sum_{1 \leq k \leq p}{kn_k}=p.\) 
\par
Lastly, an upper bound can be computed when \(\delta=1/4-\delta_1, p \leq n^{\delta_1}, n \geq n(\delta_1,c):\)
\begin{equation}\label{upperbound0}
    \mathbb{E}[tr(B^{2p})] \leq 2^{2p}L(n)n^{p+1}e^8.
\end{equation}
From above,
\[\mathbb{E}[tr(B^{2p})] \leq nL(n)C_pp!\sum_{(n_1,\hspace{0.05cm} ...\hspace{0.05cm} ,n_p)}{n^{\sum_{1 \leq k \leq p}{n_k}+(1/2-2\delta_1)(p-n_1)}\prod_{1 \leq k \leq p}{\frac{1}{(k!)^{n_k}n_k!}}\prod_{2 \leq k \leq p}{(2k)^{kn_k}}} =\]
\[=nL(n)n^{(1/2-2\delta_1)p}C_pp!\sum_{(n_1,\hspace{0.05cm} ...\hspace{0.05cm} ,n_p)}{n^{(1/2+2\delta_1)n_1+\sum_{2 \leq k \leq p}{n_k}}\prod_{1 \leq k \leq p}{\frac{1}{(k!)^{n_k}n_k!}}\prod_{2 \leq k \leq p}{(2k)^{kn_k}}}.\]
As \(n_1=p-\sum_{k \geq 2}{kn_k}\) and \(p! \leq n_1!p^{p-n_1}=n_1!p^{\sum_{k \geq 2}{kn_k}},\) the last expression is at most
\[nL(n)n^pC_p\sum_{(n_1,\hspace{0.05cm} ...\hspace{0.05cm} ,n_p)}{n^{\sum_{2 \leq k \leq p}{(1-k(1/2+2\delta_1))n_k}}\prod_{2 \leq k \leq p}{\frac{p^{kn_k}}{(k!)^{n_k}n_k!}}\prod_{2 \leq k \leq p}{(2k)^{kn_k}}} \leq \]
\[\leq nL(n)n^{p}C_p\sum_{(n_2,\hspace{0.05cm} ...\hspace{0.05cm} ,n_p)}{n^{-2\delta_1\sum_{2 \leq k \leq p}{kn_k}}\prod_{2 \leq k \leq p}{\frac{p^{kn_k}}{(k!)^{n_k}n_k!}} \prod_{2 \leq k \leq p}{(2k)^{kn_k}}}\]
employing \(\frac{1-k(1/2+2\delta_1)}{k} \leq \frac{1-2(1/2+2\delta_1)}{2}=-2\delta_1\) for \(k \geq 2.\) Since \(C_p \leq 2^{2p}, k! \geq (ke^{-1})^k,\) and \(2epn^{-2\delta_1} \leq 2en^{-\delta_1} \leq 1,\) the above sum is upper bounded by
\[2^{2p}L(n)n^{p+1}\sum_{(n_2,\hspace{0.05cm} ...\hspace{0.05cm} ,n_p)}{\prod_{2 \leq k \leq p}{\frac{(pn^{-2\delta_1})^{kn_k}(2k)^{kn_k}}{(ke^{-1})^{kn_k}n_k!}}}=2^{2p}L(n)n^{p+1}\sum_{(n_2,\hspace{0.05cm} ...\hspace{0.05cm} ,n_p)}{\prod_{2 \leq k \leq p}{\frac{(2epn^{-2\delta_1})^{kn_k}}{n_k!}}} \leq \]
\[\leq 2^{2p}L(n)n^{p+1}\sum_{(n_2,\hspace{0.05cm} ...\hspace{0.05cm} ,n_p)}{\prod_{2 \leq k \leq p}{\frac{(2epn^{-2\delta_1})^{n_k}}{n_k!}}} \leq  2^{2p}L(n)n^{p+1}\exp(\sum_{2 \leq k \leq p}{2epn^{-2\delta_1}})=\]
\[=2^{2p}L(n)n^{p+1}\exp(2ep(p-1)n^{-2\delta_1}) \leq 2^{2p}L(n)n^{p+1}e^8.\]

\subsection{Large Conditional Moments}\label{twistedmethod}

This subsection proves
\begin{equation}\label{trace}
    \mathbb{E}_*[tr((A_s+A_{B,\kappa})^{2p})-tr(A_s^{2p})] \leq 2mc(\kappa,c) \cdot (M^{2p}+n^{-\delta} (\max{(2,M)})^{2p}(2m)^{2p}(2p)^{16p^2}+s(p,M))
\end{equation}
for \(\delta=1/4-\delta_1, p \leq \sqrt{\log{n}}, m \leq n^{1/2}, n \geq n(\delta_1,c),\) and \(s:\mathbb{N} \times (0,\infty) \to [0,\infty)\) given by
\begin{equation}\label{s(p,M)function}
   s(p,M)=\sum_{1 \leq l \leq p-1}{M^{2l}\sum_{1 \leq t \leq p-l+1, 0 \leq l_0 \leq \min{(\frac{t-1}{2},l)}}{\binom{l-l_0+t-1}{l-l_0}\binom{t}{2l_0}}b_{p-l,t}},
\end{equation}
where \(\mathcal{C}(l)\) is the set of pairwise non-isomorphic even cycles of length \(2l,\) with \(n_1=l,\) and the first vertex unmarked (call two cycles \(\mathbf{i},\mathbf{j}\) of length \(2l\) \textit{isomorphic} if \(i_s=i_t \Longleftrightarrow j_s=j_t\) for all \(0 \leq s,t \leq 2l\)), and \(b_{l,t}\) is the number of vertices \(v\) of multiplicity \(t\) in \(\mathbf{i}=(i_0,i_1, \hspace{0.05cm} ... \hspace{0.05cm}, i_{2l}) \in \mathcal{C}(l):\) i.e., \(|\{0 \leq j \leq 2l, i_j=v\}|=t.\)
\par
In the classical case underlying (\ref{gentrace}), the sole contributors to the trace are the even cycles, which are in turn mapped to tuples \((n_1,n_2,\hspace{0.05cm} ...\hspace{0.05cm} ,n_p)\) with \(\sum_{1 \leq k \leq p}{kn_k}=p.\) In the current situation, this remains true, and the change of summation contains essentially one additional parameter: the non-zero entries of \(A_{B,\kappa},\) a matrix whose sparsity (encoded by \(S\)) is vital towards obtaining (\ref{trace}). 
\par
Since \(A_{\kappa}=A_{B,\kappa}+A_s,\) the left-hand side of (\ref{trace}) is a sum over cycles with contributions determined not only by their vertices, but also by whether their factors are entries of \(A_{B,\kappa}\) or \(A_s.\) Say \(a_{ij}\) \textit{belongs to} \(A_{B,\kappa}, A_s\) if \((i,j) \in S, (i,j) \not \in S,\) respectively, where \(1 \leq i,j \leq n.\) Then
\begin{equation}\label{tracemod}
    \mathbb{E}_*[tr((A_s+A_{B,\kappa})^{2p})-tr(A_s^{2p})]=n^{-p}\sum_{(i_0, i_1, \hspace{0.05cm} ... \hspace{0.05cm}, i_{2p-1})}{\mathbb{E}_*[a_{i_0i_1}a_{i_1i_2}...a_{i_{2p-1}i_{0}}]}
\end{equation}
where all the entries appearing in the product belong either to \(A_s\) or \(A_{B,\kappa},\) with at least one of them in the latter category. By independence, for any non-negative integers \((p_{ij})_{1 \leq i \leq j \leq n},\)
\[\mathbb{E}_*[\prod_{1 \leq i \leq j \leq n}{a^{p_{ij}}_{ij}}]=\prod_{1 \leq i \leq j \leq n, (i,j) \in S}{\mathbb{E}[a^{p_{ij}}_{ij} \hspace{0.05cm}| \hspace{0.05cm} \kappa \sqrt{n}<|a_{ij}| \leq M \sqrt{n}]} \cdot \prod_{1 \leq i \leq j \leq n, (i,j) \not \in S}{\mathbb{E}[a^{p_{ij}}_{ij}\chi_{|a_{ij}| \leq n^{\delta}} \hspace{0.05cm}| \hspace{0.05cm} |a_{ij}| \leq \kappa \sqrt{n}]}.\]
By symmetry, if some \(p_{ij}\) is odd, then the expectation is zero; else,
\begin{equation}\label{sub}
    \mathbb{E}_*[\prod_{1 \leq i \leq j \leq n}{a^{p_{ij}}_{ij}}] \leq c(\kappa,c)M^{\sum_{(i,j) \in S}{p_{ij}}} \cdot  \prod_{1 \leq i \leq j \leq n, (i,j) \not \in S}{\mathbb{E}[a^{p_{ij}}_{ij}\chi_{|a_{ij}| \leq n^{\delta}}]}.
\end{equation}
since \(\mathbb{P}(|a_{11}| \leq \kappa \sqrt{n}) \geq 1-2c(\kappa \sqrt{n})^{-4}.\) In other words, conditional moments can be replaced by unconditional ones for entries belonging to \(A_s,\) and by powers of \(M\) for entries belonging to \(A_{B,\kappa},\) at a cost of a multiplicative factor \(c(\kappa,c).\) This observation is used when bounding the terms on the right-hand side of (\ref{tracemod}).
\par
Keeping the terminology introduced in subsection \ref{method}, the above paragraph entails only even cycles contribute in (\ref{tracemod}). For such \(\mathbf{i},\) let \(\mathbf{i}',\mathbf{i}''\) be the strings of \(2p\) elements such that for \(0 \leq t \leq 2p-1,\) if \(a_{i_t i_{t+1}}\) belongs to \(A_s,\) then \(\mathbf{i}'_t=(i_t,i_{t+1}),\mathbf{i}''_t=\emptyset;\) else,  \(\mathbf{i}'_t=\emptyset, \mathbf{i}''_t=(i_t,i_{t+1}),\) and adopt \(\mathbf{i}=(\mathbf{i}',\mathbf{i}'')\) as a shorthand for this decomposition. Put differently, \(\mathbf{i}',\mathbf{i}''\) record the edges of \(\mathbf{i}\) belonging to \(A_s\) and \(A_{B,\kappa},\) respectively: moreover, by ignoring the empty set entries in these sequences, they can be naturally seen as subgraphs of \(\mathbf{i},\) an interpretation implicitly assumed henceforth. An important observation is that \(\mathbf{i}, \mathbf{i}', \mathbf{i}''\) share the property underlying even cycles: any undirected edge appears in each of them an even number of times (since no entry belongs to both \(A_s\) and \(A_{B,\kappa}\)). Steps \(1'-5'\) below consider the contributions of cycles \(\mathbf{i}\) for \(\mathbf{i}''\) fixed, while step \(6'\) sums them over all such directed graphs.
\par
Proceed with the first five steps: since \(\mathbf{i}''\) is fixed at this stage (denote its length by \(2l\)), the second summation in (\ref{tracemod}) is over \(\mathbf{i}'\) with \(\mathbf{i}=(\mathbf{i}',\mathbf{i}''),\) and as in the classical case, a change of summation is employed: from \(\mathbf{i}'\) to tuples \((n'_1,n'_2,\hspace{0.05cm} ...\hspace{0.05cm} ,n'_{p-l})\) with \(\sum_{1 \leq k \leq p-l}{kn'_k}=p-l.\) For any even cycle \(\mathbf{i},\) let \(N'_{\mathbf{i}}(k)\) be the set of vertices of \(\mathbf{i}\) appearing as right endpoints of marked edges of \(\mathbf{i}'\) exactly \(k\) times, and \(n'_k:=|N'_{\mathbf{i}}(k)|\) for \(1 \leq k \leq p-l\) (since \(\mathbf{i}'\) and \(\mathbf{i}''\) share no undirected edge, marking them either separately or jointly in \(\mathbf{i}\) leads to the same configuration of marked edges). In what follows, \(\mathbf{i}=(\mathbf{i}',\mathbf{i}'')\) is an even cycle with \((n'_1,n'_2,\hspace{0.05cm} ...\hspace{0.05cm} ,n'_{p-l})\) fixed and \(1 \leq l \leq p-1.\) Although steps \(1-5\) do not generally hold when \((n_1,n_2,\hspace{0.05cm} ...\hspace{0.05cm} ,n_p)\) is replaced by \((n'_1,n'_2,\hspace{0.05cm} ...\hspace{0.05cm} ,n'_{p-l}),\) they can be modified and still yield useful bounds.
\par
\underline{Step \(1'.\)} Map the marked edges of \(\mathbf{i}'\) to a Dyck path of length \(2p-2l.\) The number of such paths is at most \(C_{p-l}=\frac{1}{p-l+1}\binom{2p-2l}{p-l}.\)
\par
\underline{Step \(2'.\)} Select the order of the marked vertices in \(\mathbf{i}':\) the number of possibilities is at most
\[\frac{(p-l)!}{\prod_{1 \leq k \leq p-l}{(k!)^{n'_k}}} \cdot \frac{1}{\prod_{1 \leq k \leq p-l}{n'_k!}}.\]
\par
\underline{Step \(3'.\)} Choose the distinct vertices of \(\mathbf{i}',\) \(V(\mathbf{i}'),\) one at a time by reading its edges in order: the number of possibilities is at most 
\[2m \cdot n^{\sum_{1 \leq k \leq p-l}{n'_k}}.\] 
Each vertex of \(\mathbf{i}'\) is \(i_0,\) marked at least once, or some endpoint of an edge in \(\mathbf{i}'';\) hence, only the first two categories, whose union has size at most \(1+\sum_{1 \leq k \leq p-l}{n'_k},\) are yet to be chosen. Since \(l>0,\) let \(v=i_t\) be the first vertex in \(\mathbf{i}\) appearing also in \(\mathbf{i}''\) (i.e., \(\mathbf{i}''\) contains some edge adjacent to \(v,\) and \(t\) is minimal). If \(v=i_0,\) then it can be chosen in at most \(2m\) ways (\(uv \in \mathbf{i}''\) yields \((\min{(u,v)},\max{(u,v)}) \in S,\) and \(|S|\leq m\)), and the desired bound follows. Otherwise, \((u,v):=(i_{t-1},i_t)\) is marked in \(\mathbf{i}'\) (by the definitions of \(t\) and \(v,\) this edge belongs to \(\mathbf{i}'\) and contains the first apparition of \(v\) in \(\mathbf{i}\)), and so there are at most \(2m  \cdot n^{(1+\sum_{1 \leq k \leq p-l}{n'_k})-1}\) possibilities (\(v\) is both an element of a fixed set of size at most \(2m,\) and of \(N'_{\mathbf{i}}(k),\) for some \(1 \leq k \leq p-l\)). 
\par
\underline{Step \(4'.\)} Choose the remaining vertices of \(\mathbf{i}'\) among the ones selected in step \(3'.\) At this stage, \((i_0,i_1)\) is fully determined since it is either in \(\mathbf{i}''\) or marked in \(\mathbf{i}'.\) The same rationale as in step \(4\) shows only the right endpoints of the unmarked edges of \(\mathbf{i}'\) have yet to be chosen, which can be done in at most \(((2l+2)!)^{4l}\prod_{2 \leq k \leq p-l}{(2k+2l)^{kn'_k}}\) ways: 

\begin{lemma}\label{lemma1'}
If \(v \in N'_{\mathbf{i}}(k),\) then the number of unmarked edges of the form \((v,u)\) is at most
\(\begin{cases} 
2k+2l, & k>1\\
1, & k=1\\
2l+2, & v \in \mathcal{E}(\mathbf{i})
\end{cases},
\)
where \(|\mathcal{E}(\mathbf{i})| \leq 4l.\)
\end{lemma}

\begin{proof}
If \(v \in N'_{\mathbf{i}}(k),\) then \(v \in N_{\mathbf{i}}(k')\) with \(k \leq k' \leq k+l\) (\(\mathbf{i''}\) contains \(l\) marked edges). Lemma~\ref{lemma1} then gives the result above for \(k \geq 1.\) For \(k=1,\) there is at most one possibility unless there exists an edge in \(\mathbf{i}''\) containing \(v\) (else, the proof of Lemma~\ref{lemma1} for this case is still valid): let \(\mathcal{E}(\mathbf{i})\) be the set of such vertices. For \(v \in \mathcal{E}(\mathbf{i}),\) there are at most \(2(l+1)\) such edges, and \(|\mathcal{E}(\mathbf{i})| \leq 2 \cdot 2l=4l\) (\(\mathbf{i}''\) contains \(2l\) edges).
\end{proof}

\par
\underline{Step \(5'.\)} Let
\(E'=|\{e \in \mathbf{i}': k(e)>2\}|.\) For \(a=a_{11}\chi_{|a_{11}| \leq n^{\delta}}, q \in \mathbb{N}, q \geq 2,\)
\[\mathbb{E}[a^2] \leq 1, \hspace{0.2cm} \mathbb{E}[a^{2q}] \leq c(\delta,c)\log{n} \cdot n^{\delta(2q-4)}:=L(n)n^{\delta(2q-4)},\]  
and so
\[\mathbb{E}[a_{\mathbf{i}'}] \leq L(n)^{(p-l)/2}\prod_{e \in \mathbf{i}', k(e)>2}{n^{\delta(2k(e)-4)}}=L(n)^{(p-l)/2}n^{2\delta(\sum_{e \in \mathbf{i}', k(e)>2}{k(e)}-2E')}\]
as there are at most \((p-l)/2\) pairwise distinct undirected edges in \(\mathbf{i}',\) each appearing at least four times in \(\mathbf{i}'.\) Because every edge \(e=uv\) with \(k(e)>2\) has either \(u \in N'_{\mathbf{i}}(k)\) or \(v \in N'_{\mathbf{i}}(k)\) for some \(k \geq 2,\)
\[\sum_{e \in \mathbf{i}', k(e)>2}{k(e)} \leq E'+\sum_{k \geq 2}{kn'_k},\]
providing
\[\mathbb{E}[a_{\mathbf{i}'}] \leq L(n)^{(p-l)/2}n^{2\delta(\sum_{k \geq 2}{kn'_k}-E')} \leq L(n)^{(p-l)/2}n^{2\delta\sum_{k \geq 2}{kn'_k}}.\]
Furthermore, an overall saving of some power of \(n\) is possible unless \(\mathbf{i}\) has a very special form.

\begin{lemma}\label{goodcycles}
For any even cycle \(\mathbf{i}\) with \(l<p,\) at least one of the following occurs:
\par
\((I)\) a factor of \(n^{2\delta}\) can be saved in step \(5':\)
\[\mathbb{E}[a_{\mathbf{i}'}] \leq L(n)^{(p-l)/2}n^{2\delta(\sum_{k \geq 2}{kn'_k}-1)},\]
\par
\((II)\) a factor of \(n/(2m)\) can be saved in step \(3':\)
\[(2m)^2 \cdot n^{\sum_{1 \leq k \leq p-l}{n'_k}-1},\]
\par
\((III)\) \(n'_1=p-l,\) \(i_0\) is unmarked in \(\mathbf{i}',\) \(\mathbf{i''}\) contains a unique undirected edge \(vw\) with \(v \in N'_{\mathbf{i}}(1) \cup \{i_0\}\) and \(w \ne i_0\) unmarked in \(\mathbf{i}'.\) 
\end{lemma}

\begin{proof}
Since
\[\mathbb{E}[a_{\mathbf{i}'}] \leq L(n)^{(p-l)/2}n^{2\delta(\sum_{k \geq 2}{kn'_k}-E')},\]
\((I)\) follows unless \(E'=0\) and \(n'_1=p-l:\) if \(E'>0,\) then it is clear; if \(E'=0, n'_1 \ne p-l,\) then \(\sum_{k \geq 2}{n'_k} \geq 1\) and the desired inequality holds too. What is left is the case \(E'=0, n'_1=p-l.\) 
\par
If \(\mathbf{i}''\) contains at least two distinct undirected edges, then \((II)\) holds. Suppose the condition is satisfied. Without loss of generality, assume each cluster of edges in \(\mathbf{i}''\) has size at most one (such a block is fully determined by its length and its first edge because any two distinct undirected edges of \(\mathbf{i}''\) share no vertex; therefore, if the first edge of the cluster is \((u,v),\) then its edges are \((u,v),(v,u),(u,v), \hspace{0.02cm} ... \hspace{0.02cm},\) which can be compressed to \(u, (u,v)\) for even, odd length, respectively, without affecting either the edges or the vertices of \(\mathbf{i}'\)). Each undirected edge \(e=vw\) of \(\mathbf{i}''\) is adjacent to either a marked vertex in \(\mathbf{i}'\) or \(i_0:\) take \(s=\min{\{0 \leq t \leq 2p, i_t \in \{v,w\}\}};\) if \(s=0,\) then \(e\) is adjacent to \(i_0;\) else, \(s>0,\) and \((i_{s-1},i_s) \in \mathbf{i}',\) because otherwise \(i_{s-1} \in \{v,w\},\) and is marked since \(i_s\) is the first apparition of a vertex in \(\mathbf{i}.\) This observation implies \((II).\)
\par
If \(\mathbf{i}''\) contains solely one undirected edge \(vw,\) then \((II)\) holds unless \((III)\) is satisfied. In this case, if \(i_0\) is marked in \(\mathbf{i}',\) then \((II)\) holds since \(\sum_{k \geq 1}{n'_k}\) can replace \(1+\sum_{k \geq 1}{n'_k}\) in step \(3'.\) If after compressing the clusters either both \(v\) and \(w\) are marked or one is \(i_0\) and the other marked, then again some saving is possible and \((II)\) ensues. Else, \((III)\) holds using \(\{v,w\} \cap (N'_{\mathbf{i}}(1) \cup \{i_0\}) \ne \emptyset.\)
\end{proof}

\par
In conclusion, merging steps \(3'\) and \(5'\) yields the overall contribution of cycles of type \((I)\) and \((II)\) is at most
\[2m \cdot n^{2\delta\sum_{k \geq 2}{kn'_k}+\sum_{k \geq 1}{n'_k}} L(n)^{(p-l)/2} \cdot (n^{-2\delta}+2m \cdot n^{-1}) \leq 4m \cdot n^{-\delta}\cdot n^{2\delta\sum_{k \geq 2}{kn'_k}+\sum_{k \geq 1}{n'_k}}.\]
for \(p \leq \sqrt{\log{n}}\) and \(n\) large enough.
\par
Putting steps \(1'-5'\) together, the computations at the end of the previous subsection can be used with the substitutions \(p \to p-l, n'_k \to n_k, 2k \to 2k+2l,\) and \(\delta=1/4-\delta_1, p \leq \sqrt{\log{n}}\) (\(8ep^2n^{-2\delta_1}\) replaces \(2epn^{-2\delta_1}\) from \(\frac{2k+2l}{ke^{-1}} \leq 4pe\)). Thus, this sum is upper bounded by
\begin{equation}\label{contri''}
     c(\kappa,c) \cdot 4m n^{-\delta} 2^{2p-2l}((2l+2)!)^{4l}e^{16} 
\end{equation}
for \(p \leq \sqrt{\log{n}}\) and  \(n\) sufficiently large.
\par
\underline{Step \(6'.\)} The conditional expectation coming from \(\mathbf{i}''\) is not larger than \(M^{2l},\) and there are at most \(\binom{2p}{2l}(2m)^{2l}\) such directed cycles for \(1 \leq l \leq p-1,\) and \(2|S| \leq 2m\) for \(l=p\) (cycles with all edges belonging to \(A_{B,\kappa}\) are fully determined by their first edge). Hence, using (\ref{contri''}) and (\ref{sub}), the overall contribution of cycles of types \((I)\) and \((II)\) is at most
\begin{equation}\label{typeIandII}
    2m \cdot M^{2p}+c(\kappa,c) \cdot 4m n^{-\delta}e^{16}\sum_{1 \leq l \leq p-1}{\binom{2p}{2l}(2m)^{2l}} \cdot 2^{2p-2l}((2l+2)!)^{4l}M^{2l}.
\end{equation}
\par
Consider now the cycles of type \((III):\) they generate a term less or equal than
\begin{equation}\label{typeIII}
   c(\kappa,c) \cdot 2m\sum_{1 \leq l \leq p-1}{M^{2l}\sum_{1 \leq t \leq p-l+1, 0 \leq l_0 \leq \min{(\frac{t}{2},\frac{l}{2})}}{\binom{l-l_0+t-1}{l-l_0}\binom{t}{2l_0}}b_{p-l,t}}.
\end{equation}
To see this, let \(vw\) be the edge appearing in a fixed \(\mathbf{i}''\) of length \(2l.\) Map each compressed cycle of type \((III)\) to an element of \(\mathbf{j} \in \mathcal{C}(p-l)\) by replacing each cluster \(\{v,w,(v,w),(w,v)\}\) by a new vertex \(\rho\) (note \(\rho\) is marked in \(\mathbf{j}\) exactly when \(v\) and \(w\) are in \(\mathbf{i}':\) hence, this procedure generates cycles in \(\mathcal{C}(p-l)\)). It is shown next that the preimage of any \(\mathbf{j} \in \mathcal{C}(p-l)\) with \(\rho\) a fixed vertex in it of multiplicity \(t,\) and the first cluster of \(\mathbf{i}\) containing \(v,\) has size in the following interval
\[[\sum_{0 \leq l_0 \leq \min{(\frac{t-1}{2},l)}}{\binom{l-l_0+t-1}{l-l_0}\binom{t-1}{2l_0}},\sum_{0 \leq l_0 \leq \min{(\frac{t}{2},l)}}{\binom{l-l_0+t-1}{l-l_0}\binom{t}{2l_0}}],\]
whereby (\ref{typeIII}) is fully justified since any element of \(\mathcal{C}(l)\) contains \(l\) pairwise distinct undirected edges (this ensues by induction and the recursive description of \(\mathcal{C}(l):\) see proof of (\ref{sizecl}) in subsection~\ref{closedformcombfunction}), which together with this mapping gives \(\mathbb{E}[a_{\mathbf{i}'}] \leq 1\) (although the upper bound suffices for (\ref{typeIII}), the lower bound comes into play in subsection~\ref{reverseineq}).
\par
Suppose first \(\rho \ne j_0,\) and let \((s_0,\rho,s_1), \hspace{0.05cm} ... \hspace{0.05cm}, (s_{2t-2},\rho,s_{2t-1})\) be the apparitions of \(\rho\) in \(\mathbf{j}\) in increasing order: \((s_2,\rho),(s_4,\rho), \hspace{0.05cm} ... \hspace{0.05cm},(s_{2t-2},\rho),(\rho,s_{2t-1})\) are unmarked and so \((\rho,s_1),(\rho,s_3), \hspace{0.05cm} ... \hspace{0.05cm} ,(\rho,s_{2t-3})\) are marked. Denote by \(2l_0\) the number of clusters of size two in \(\mathbf{i},\) an even cycle in the preimage of \(\mathbf{j}\) (this number is even because \(\mathbf{i}\) is): then \(0 \leq l_0 \leq \min{(\frac{t}{2},l)}\) (each contains an edge of \(\mathbf{i}''\)) and thus, the preimage has at most \(\binom{l-l_0+t-1}{l-l_0}\binom{t}{2l_0}\) elements since once the sizes of the clusters underlying \(\rho\) are fixed, by induction on \(1 \leq k \leq t,\) the \(k^{th}\) cluster \((s_{2k-2},\rho,s_{2k-1})\) is fully determined: for \(k=1,\) it is clear as \(\rho \in \{v,(v,w)\}\) and has fixed size; for \(k \geq 2,\) the counterpart of \((s_{2k-2},\rho)\) has already been decided, yielding the first vertex of the cluster, and its size dictates whether it has a second vertex or not, which is then fully determined by the first. Lastly, choosing tuples of even integers \((2x_1,2x_2, \hspace{0.05cm} ... \hspace{0.05cm}, 2x_t)\) with \(x_1+...+x_t=l-l_0, x_j \geq 0\) can be done in \(\binom{t+l-l_0-1}{l-l_0}\) ways (for \(t,l \in \mathbb{N},\) let \(\tilde{a}_{t,l}=|\{(2x_1,2x_2, \hspace{0.05cm} ...\hspace{0.05cm}, 2x_t): x_1+...+x_t=l, x_j \in \mathbb{Z}, x_j \geq 0\}|;\) clearly, \(\tilde{a}_{1,l}=1,\) and for \(t \geq 2,\) \(\tilde{a}_{t,l}=\tilde{a}_{t-1,l}+\tilde{a}_{t,l-1}\) as \(x_1=0\) or \(x_1>0;\) this yields by induction on \(t+l, \tilde{a}_{t,l}=\binom{t+l-1}{l}\)). 
\par
Conversely, select \(2l_0\) elements out of \(\{1,2,\hspace{0.05cm} ... \hspace{0.05cm}, t-1\}\) with \(0 \leq l_0 \leq \min{(\frac{t-1}{2},l)}:\) call this set \(D.\) Then there is an even cycle \(\mathbf{i}\) in the preimage of \(\mathbf{j}\) with the \(d^{th}\) apparition of \(\rho\) replaced by two vertices for \(d \in D\) and by one if \(d \not \in D, d \ne t.\) The last cluster is fully determined by the previous ones, while for the rest there are two possibilities: traverse the clusters from left to right; for the \(k^{th},\) the first edge, \((s_{2k-2},\rho)\) is fixed (it is unmarked), while the second can be chosen in two ways because it is marked, generating two scenarios for \(\rho,\) in which it has size one and two, respectively. Thus, it is always possible to decide the size of any of its apparitions, and \(|D|\) even ensures the cycle obtained is also even. Lastly, the case \(\rho=j_0\) is analogous to \(\rho \ne j_0,\) the main difference being that the apparitions of \(\rho\) are \((\rho,s_1), (s_2,\rho,s_3), \hspace{0.05cm} .... \hspace{0.05cm}, (s_{2t-2},\rho,s_{2t-1}),(s_{2t},\rho).\) Finally, (\ref{trace}) ensues from (\ref{typeIandII}) and (\ref{typeIII}).

\subsection{Asymptotics of Large Conditional Moments}\label{closedformcombfunction}

This subsection justifies for \(M>0,\)
\begin{equation}\label{simplifiedf}
    \lim_{p \to \infty}{s(p,M)^{1/2p}}=\begin{cases}
    2, & M \leq 1 \\
    \frac{M^2+1}{M}, & M \geq 1 \end{cases}.
\end{equation}
\par
Recall that
\[s(p,M)=\sum_{1 \leq l \leq p-1}{M^{2l}\sum_{1 \leq t \leq p-l+1, 0 \leq l_0 \leq \min{(\frac{t}{2},l)}}{\binom{l-l_0+t-1}{l-l_0}\binom{t}{2l_0}}b_{p-l,t}}.\]
For \(t \geq l+2\) \(b_{l,t}=0\) (any cycle in \(\mathcal{C}(l)\) contains \(2l+1\) vertices, out of which at least \(l+1\) are pairwise distinct: \(i_0\) and the marked vertices), and \(b_{l,l+1}=1:\) the sole element of \(\mathcal{C}(l)\) with a vertex repeated \(l+1\) times is \((u_0,u_1,u_0, \hspace{0.05cm} ... \hspace{0.05cm}, u_0, u_l, u_0)\) with \(u_0, u_1,\hspace{0.05cm} ... \hspace{0.05cm} , u_l\) pairwise distinct (\(l\) vertices have multiplicity \(1,\) and one has multiplicity \(l+1;\) since the first vertex appears at least twice, its multiplicity is \(l+1,\) and all the other vertices show up exactly once, yielding a unique cycle up to isomorphism).
\par
(\ref{simplifiedf}) is concluded in two phases, first, computing the asymptotic behavior of \((b_{l,t}):\)
\begin{equation}\label{bsizes}
    \frac{1}{4l} \cdot \binom{2l+1-t}{l} \leq b_{l,t} \leq  (l+1)^{120} \cdot \binom{2l+1-t}{l}
\end{equation}
for \(1 \leq t \leq l+1,\) and second, using binomial proxies for \(b_{l,t}\) in \(s(p,M)\) (any polynomial factor becomes negligible when \(p \to \infty\)).
\par
Proceed with understanding the sizes of \((b_{l,t})_{1 \leq l \leq t \leq l+1},\) for which two recursions, describing the sequence itself and \((\mathcal{C}(l))_{l \geq 1},\) respectively, are essential. Observe that
\begin{equation}\label{sizecl}
    |\mathcal{C}(l)|=C_l:
\end{equation}
on the one hand, steps \(1, 2, 4\) from subsection~\ref{method} yield \(|\mathcal{C}(l)| \leq C_l.\) On the other hand, recall the recursive characterization of the Catalan numbers: 
\[C_{l+1}=\sum_{0 \leq a \leq l}{C_aC_{l-a}}, \hspace{0.2cm} C_0=C_1=1.\]
Since \(|\mathcal{C}(1)|=1,\) justifying
\[|\mathcal{C}(l+1)| \geq 2|\mathcal{C}(l)|+\sum_{1 \leq a \leq l-1}{|\mathcal{C}(a)| \cdot |\mathcal{C}(l-a)|}\]
yields by induction \(|\mathcal{C}(l)| \geq C_l.\) To do so, construct three types of cycles among the elements of \(\mathcal{C}(l+1):\)
\par
\((i)\) \(\mathbf{i}=(v_0,v_1, \hspace{0.05cm} ... \hspace{0.05cm}, v_{2l},v_0) \in \mathcal{C}(l)\) with an extra loop at \(v_0: (v_0,u,v_0,v_1,\hspace{0.05cm} ... \hspace{0.05cm}, v_{2l},v_0)\) and \(u\) new (i.e., not among the vertices of \(\mathbf{i}\));
\par
\((ii)\) \(\mathbf{i}=(v_0,v_1, \hspace{0.05cm} ... \hspace{0.05cm}, v_{2l},v_0) \in \mathcal{C}(l)\) with an extra loop at \(v_1: (v_0,v_1,u,v_1,v_2,\hspace{0.05cm} ... \hspace{0.05cm}, v_{2l},v_0)\) and \(u\) new;
\par
\((iii)\) \((u_0,u_1,u_2,S_1,u_2,u_1,S_2)\) with \((u_2,S_1,u_2) \in \mathcal{C}(a), (u_0,u_1,S_2) \in \mathcal{C}(l-a),\) no vertex appearing in both, \(u_0,u_1,u_2\) pairwise distinct, and \(1 \leq a \leq l-1. \newline\)
Clearly, these three families are pairwise disjoint (by considering the second apparitions of the first three vertices of such cycles), and their union is a subset of \(\mathcal{C}(l+1)\) of size 
\[|\mathcal{C}(l)|+|\mathcal{C}(l)|+\sum_{1 \leq a \leq l-1}{|\mathcal{C}(a)| \cdot |\mathcal{C}(l-a)|}.\]
The proof of (\ref{sizecl}) is thus complete, and its crucial by-product is: 
\begin{center}
     \(\mathcal{C}(l+1)\) consists of three components: \((i), (ii),\) and \((iii).\)
\end{center}
\par
Returning to (\ref{bsizes}), (\ref{sizecl}) yields the desired chain of inequalities for \(t=1:\)
\[C_{l-1}=|\{(v_0,u,v_0,v_1,\hspace{0.05cm} ... \hspace{0.05cm}, v_{2l},v_0), (v_0,v_1, \hspace{0.05cm} ... \hspace{0.05cm}, v_{2l},v_0) \in \mathcal{C}(l-1)\}|
\leq b_{l,1} \leq \sum_{1 \leq t \leq l+1}{b_{l,t}}=(2l+1)C_l,\]
from which
\[\frac{1}{4l-2} \cdot \binom{2l}{l} \leq b_{l,1} \leq \frac{2l+1}{l+1} \cdot \binom{2l}{l}.\]
\par
Let now \(2 \leq t \leq l+1,\) and denote by \(f_{l,t}\) the number of elements of \(\mathcal{C}(l)\) in which the first vertex appears exactly \(t\) times; similarly, \(s_{l,t}\) is the number of elements in \(\mathcal{C}(l)\) whose second vertex has multiplicity \(t.\)
The aforementioned description of \(\mathcal{C}(l+1)\) gives
\[b_{l+1,t}=(f_{l,t-1}+b_{l,t}-f_{l,t})+(s_{l,t-1}+b_{l,t}-s_{l,t})+\sum_{1 \leq a \leq l-1}{[C_{l-a}b_{a,t}+ C_{a}(b_{l-a,t}-s_{l-a,t}+s_{l-a,t-1})]}\]
since a vertex \(v\) has multiplicity \(t>1\) if and only if
\(\newline (i)\) its multiplicity in \(\mathbf{i}\) is \(\begin{cases} 
t-1, & v=v_0\\
t, & v \ne v_0
\end{cases};\)
\(\hspace{0.5cm} (ii)\) its multiplicity in \(\mathbf{i}\) is \(\begin{cases} 
t-1, & v=v_1\\
t, & v \ne v_1
\end{cases};\)
\(\newline (iii)\) \(v \in (u_2,S_1,u_2)\) has multiplicity \(t,\) or \(v \in (u_0,u_1,S_1)\) has multiplicity \(\begin{cases} 
t-1, & v=u_1\\
t, & v \ne u_1
\end{cases}.\newline\)
By rearranging terms,
\begin{equation}\label{brecursion}
    b_{l+1,t}=2\sum_{0 \leq a \leq l-1}{C_{a}b_{l-a,t}}+
    (f_{l,t-1}-f_{l,t})+(s_{l,t-1}-s_{l,t}) +\sum_{1 \leq a \leq l-1}{ C_{a}(s_{l-a,t-1}-s_{l-a,t})}.
\end{equation}

\begin{lemma}\label{lemmafands}
    For \(k \geq -1\) and \(l \geq k+1,\)
    \[f_{l,l-k}=\binom{l+k}{k+1}-\binom{l+k}{k}.\]
\end{lemma}

\begin{proof}
Reasoning as above, the description of \(\mathcal{C}(l+1)\) gives for \(2 \leq t \leq l+1,\)
\begin{equation}\label{recursionf}
    f_{l+1,t}=f_{l,t-1}+f_{l,t}+\sum_{1 \leq a \leq l-1}{C_af_{l-a,t}},
\end{equation}
where \(f_{l,t}=0\) if \(t>l+1.\) Proceed by induction on \(k \geq -1\) for \(l \geq k+1.\) The base case \(k=-1\) is immediate from \(f_{l,l+1}=b_{l,l+1}=1.\) Let \(k \geq 0\) be fixed.
\par
Use induction now on \(l \geq k+1:\) the base case \(l=k+1\) is clear inasmuch as the first vertex of any even cycle has multiplicity at least two. Take next \(l \geq k+2:\) since
\begin{equation}\label{eqf}
    \binom{l+k+1}{k+1}-\binom{l+k+1}{k}=f_{l,l-k}+f_{l,l-k+1}+\binom{l+k-1}{k-1}-\binom{l+k-1}{k-2}
\end{equation}
from
\[\binom{l+k+1}{k+1}-\binom{l+k+1}{k}=\binom{l+k}{k+1}+\binom{l+k}{k}-\binom{l+k}{k}-\binom{l+k}{k-1}=\binom{l+k}{k+1}-\binom{l+k}{k-1}=\]
\[=f_{l,l-k}+\binom{l+k}{k}-\binom{l+k}{k-1}=f_{l,l-k}+f_{l,l-k+1}+\binom{l+k}{k}-\binom{l+k}{k-1}-\binom{l+k-1}{k}+\binom{l+k-1}{k-1}=\]
\[=f_{l,l-k}+f_{l,l-k+1}+\binom{l+k-1}{k-1}-\binom{l+k-1}{k-2},\]
in light of (\ref{recursionf}) rewritten for \(l \geq k+1\) as
\[f_{l+1,l-k+1}=f_{l,l-k}+f_{l,l-k+1}+\sum_{1 \leq a \leq k}{C_af_{l-a,l-k+1}},\]
it suffices to prove that for \(l \geq k+1,\)
\[A(l,k):=\binom{l+k-1}{k-1}-\binom{l+k-1}{k-2}=\sum_{1 \leq a \leq k}{C_a(\binom{l+k-1-2a}{k-a}-\binom{l+k-1-2a}{k-a-1})}:=B(l,k).\] 
Because
\[A(l,k)=A(l-1,k)+A(l-1,k-1), \hspace{0.2cm} B(l,k)=B(l-1,k)+B(l-1,k-1),\]
(in \(B(l,k)-B(l-1,k)\) the summation is over \(1 \leq a \leq k-1\) since the coefficient of \(C_{k}\) in \(B(l,k)\) is \(1\))
from \(\binom{n}{m}=\binom{n-1}{m}+\binom{n-1}{m-1},\) showing \(A(l,k)=B(l,k)\) for \(l=k+1\) is sufficient as induction on \(l-k \geq 1\) yields the desired identity (for \(d=l-k\) fixed, induction is used anew on \(l \geq k+1,\) whose base case \(l=d+1,k=1\) is immediate because \(f_{d+1,1}=0=\binom{2d+1}{d+1}-\binom{2d+1}{d}\)). In this situation, the desired result is
\[\binom{2k}{k-1}-\binom{2k}{k-2}=\sum_{1 \leq a \leq k}{C_a(\binom{2k-2a}{k-a}-\binom{2k-2a}{k-a-1})}.\]
Note the right-hand side term is
\[\sum_{1 \leq a \leq k}{C_aC_{k-a}}=C_{k+1}-C_k=C_k \cdot \frac{3k}{k+2},\]
and
\[\binom{2k}{k-1}-\binom{2k}{k-2}= C_k \cdot \frac{3k}{k+2}=\binom{2k}{k} \cdot \frac{3k}{(k+2)(k+1)}\]
since simplifying this equation by \(\frac{k!(k-1)!}{(2k)!}\) turns it equivalent to
\[\frac{1}{k+1}-\frac{k-1}{(k+1)(k+2)}=\frac{1}{k} \cdot \frac{3k}{(k+2)(k+1)}=\frac{3}{(k+1)(k+2)}.\]
\end{proof}

\par
Notice \(s\) and \(f\) are comparable, i.e., 
\begin{equation}\label{squeezesbetweenf}
    f_{l-1,t} \leq s_{l,t} \leq f_{l,t+1}:
\end{equation}
if \(\mathbf{i} \in \mathcal{C}(l-1)\) has the first vertex of multiplicity \(t,\) then \((u_0,\mathbf{i},u_0) \in \mathcal{C}(l)\) has the second vertex of multiplicity \(t,\) implying the first inequality. The upper bound ensues by induction on \(t+l\) from (\ref{recursionf}): using the description of \(\mathcal{C}(l+1),\) \(s_{l,1}=C_{l-1}=f_{l,2},\) and for \(2 \leq t \leq l+1,\) 
\[s_{l+1,t}=s_{l,t-1}+\sum_{1 \leq a \leq l-1}{C_as_{l-a,t}}.\]
\par
Now the growth of \((b_{l,t})_{2 \leq t \leq l+1}\) can be fully established.

\begin{lemma}\label{lemma10}
For \(2 \leq t \leq l+1,\) 
\[f_{l,t} \leq b_{l,t} \leq (l+1)^{120} \cdot f_{l,t}.\]
\end{lemma}

\begin{proof}
The lower bound is evident. Focus next on the second inequality. Recall (\ref{brecursion}):
\[b_{l+1,t}=2\sum_{0 \leq a \leq l-1}{C_{a}b_{l-a,t}}+
(f_{l,t-1}-f_{l,t})+(s_{l,t-1}-s_{l,t})+\sum_{1 \leq a \leq l-1}{ C_{a}(s_{l-a,t-1}-s_{l-a,t})}.\]
For the time being, let \(\Sigma_s=\sum_{1 \leq a \leq l-1}{ C_{a}(s_{l-a,t-1}-s_{l-a,t})}.\) Subtract \(f_{l+1,t}\) from both sides
\[b_{l+1,t}-f_{l+1,t}=2\sum_{0 \leq a \leq l-1}{C_{a}b_{l-a,t}}+
(f_{l,t-1}-f_{l,t})+(s_{l,t-1}-s_{l,t})+\Sigma_s-f_{l+1,t},\]
use (\ref{recursionf})
\[b_{l+1,t}-f_{l+1,t}=2\sum_{0 \leq a \leq l-1}{C_{a}b_{l-a,t}}
-2f_{l,t}-\sum_{1 \leq a \leq l-1}{C_af_{l-a,t}}+(s_{l,t-1}-s_{l,t})+\Sigma_s,\]
change \(b_{l-a,t}\) to \(b_{l-a,t}-f_{l-a,t}\)
\[b_{l+1,t}-f_{l+1,t}=2\sum_{0 \leq a \leq l-1}{C_{a}(b_{l-a,t}-f_{l-a,t})}
+\sum_{1 \leq a \leq l-1}{C_af_{l-a,t}}+(s_{l,t-1}-s_{l,t}) +\Sigma_s,\]
dispense with the second summation
\[b_{l+1,t}-f_{l+1,t}=2\sum_{0 \leq a \leq l-1}{C_{a}(b_{l-a,t}-f_{l-a,t})}
+(f_{l+1,t}-f_{l,t-1}-f_{l,t})+(s_{l,t-1}-s_{l,t})+\Sigma_s.\]
\par
Consider now \(\Sigma_s:\) (\ref{squeezesbetweenf}) implies
\[\sum_{1 \leq a \leq l-1}{C_{a}(s_{l-a,t-1}-s_{l-a,t})} \leq \sum_{1 \leq a \leq l-1}{C_{a}f_{l-a,t}} \leq f_{l+1,t},\]
whereby \(b_{l+1,t}-f_{l+1,t}\) is at most
\[2\sum_{0 \leq a \leq l-1}{C_{a}(b_{l-a,t}-f_{l-a,t})}+(f_{l+1,t}-f_{l,t}-f_{l,t-1})
+(f_{l,t}-f_{l-1,t-1})+f_{l+1,t},\]
from which
\begin{equation}\label{bfineq}
    b_{l+1,t}-f_{l+1,t} \leq 2\sum_{0 \leq a \leq l-1}{C_{a}(b_{l-a,t}-f_{l-a,t})}+2f_{l+1,t}.
\end{equation}
\par
Induction on \(k \geq -1\) entails
\[b_{l,l-k}-f_{l,l-k} \leq (l+1)^{120} \cdot f_{l,l-k}.\]
When \(k=-1,\) \(b_{l,l+1}-f_{l,l+1}=0.\) Let now \(k \geq 0\) be fixed, and suppose \(l \geq 4\) (if \(l \leq 3,\) then \(b_{l,l-k} \leq (2l+1)C_l<2^{120}\)). In light of (\ref{bfineq}) with \(t=l-k+1 \geq 2,\) it suffices to show
\[2\sum_{0 \leq a \leq l-1}{C_{a}(l-a+1)^{120}f_{l-a,t}}+2f_{l+1,t} \leq (l+2)^{120} \cdot f_{l+1,t}:\]
to see this, note that
\[\frac{f_{l,l-k+1}}{f_{l+1,l-k+1}}=\frac{k+1}{k} \cdot \frac{\binom{l+k-1}{k-1}}{\binom{l+k+1}{k}}=\frac{(k+1)(l+1)}{(l+k)(l+k+1)} \leq \frac{l+1}{(\sqrt{l-1}+\sqrt{l})^2} \leq \frac{1}{4}+\frac{2}{l},\]
employing \(\frac{x}{(x+\alpha)(x+\beta)} \leq \frac{1}{(\sqrt{\alpha}+\sqrt{\beta})^2}\) for \(x,\alpha,\beta>0.\) This inequality implies for \(0 \leq a \leq \frac{l}{2}-1\)
\[\frac{f_{l-a,l-k+1}}{f_{l+1,l-k+1}} \leq (\frac{1}{4}+\frac{2}{l})...(\frac{1}{4}+\frac{2}{l-a}) \leq (\frac{1}{4}+\frac{4}{l+2})^{a+1},\]
from which
\[2\sum_{0 \leq a \leq l/2-1}{C_{a}(l-a+1)^{120}f_{l-a,t}} \leq 2(l+1)^{120} \cdot f_{l+1,t} \sum_{0 \leq a \leq l/2-1}{(\frac{1}{4}+\frac{4}{l+2})^{a+1}} \leq\]
\[\leq 2(l+1)^{120} \cdot f_{l+1,t} \cdot \frac{1/4+4/(l+2)}{3/4-4/(l+2)}=2(l+1)^{120} \cdot f_{l+1,t} \cdot \frac{l+18}{3l-10},\]
while
\[2\sum_{l/2-1< a \leq l-1}{C_{a}(l-a+1)^{120}f_{l-a,t}} \leq 2(l/2+2)^{120} \cdot (f_{l+1,t}-f_{l,t-1}-f_{l,t}) \leq 2(l/2+2)^{120} \cdot f_{l+1,t}.\]
The claim follows from
\[2(\frac{l+1}{l+2})^{120} \cdot \frac{l+18}{3l-10}+(\frac{l+4}{2l+4})^{120} \cdot 2+\frac{2}{(l+2)^{120}} \leq 1:\]
the left-hand side is at most
\[(\frac{5}{6})^{120} \cdot 22+(\frac{2}{3})^{120} \cdot 2+\frac{2}{6^{120}} \leq 25 \cdot (\frac{5}{6})^{120} \leq \frac{25}{1+120 \cdot 1/5}=1.\]
\end{proof}

\par
To finalize (\ref{bsizes}), note the result is clear for \(t=l+1,\) and for \(t=l-k+1, 1 \leq k \leq l-1,\)
\[f_{l,l-k+1}=\frac{l-k}{k}\binom{l+k-1}{k-1}=\frac{l-k}{l+k}\binom{l+k}{l} \in [\frac{1}{2l}\binom{l+k}{l},\binom{l+k}{l}],\]
which in conjunction with Lemma~\ref{lemma10} entails the desired inequalities.
\par
Having completed (\ref{bsizes}), proceed with the last missing piece of this subsection, (\ref{simplifiedf}). Let \(x!:=(\lfloor{} x \rfloor)!\) for \(x \geq 0:\) Stirling's formula yields
\[C_1\sqrt{n} \cdot (\frac{n}{e})^n \leq n! \leq C_2\sqrt{n} \cdot (\frac{n}{e})^n\]
for universal constants \(C_1,C_2>0\) and all \(n \geq 0,\) from which
\[\lim_{p \to \infty}{s(p,M)^{1/p}}=\sup{M^{2x}\frac{(x-y+z)^{x-y+z}}{(x-y)^{x-y}(2y)^{2y}(z-2y)^{z-2y}} \cdot \frac{(2-2x-z)^{2-2x-z}}{(1-x)^{1-x}(1-x-z)^{1-x-z}}}\]
over \(0 \leq x \leq 1, 0 \leq z \leq 1-x, 0 \leq y \leq \min{(\frac{z}{2},x)}:\) take \(l=px,l_0=py,t=pz,\) and so
\[(M^{2l}\binom{l-l_0+t}{l-l_0}\binom{t}{2l_0}\binom{2p-2l-t}{p-l})^{1/p}=(M^{2px}\binom{px-py+pz}{pz}\binom{pz}{2py}\binom{2p-2px-pz}{p-px})^{1/p}=\]
\[=M^{2x}\frac{(x-y+z)^{x-y+z}}{(x-y)^{x-y}(2y)^{2y}(z-2y)^{z-2y}} \cdot \frac{(2-2x-z)^{2-2x-z}}{(1-x)^{1-x}(1-x-z)^{1-x-z}}+o(M)\]
with the last term tending to \(0\) as \(p \to \infty\) uniformly in \(x,y,z.\) (\ref{simplifiedf}) is a consequence of:

\begin{lemma}
For \(M>0\) fixed, let \(h:D \to (0,\infty),\)
\[h(x,y,z)=M^{2x} \cdot \frac{(x-y+z)^{x-y+z}}{(x-y)^{x-y}(2y)^{2y}(z-2y)^{z-2y}} \cdot \frac{(2-2x-z)^{2-2x-z}}{(1-x)^{1-x}(1-x-z)^{1-x-z}}\]
where \(D=\{(x,y,z) \in \mathbb{R}^3 :0 \leq x \leq 1, 0 \leq z \leq 1-x, 0 \leq y \leq \min{(\frac{z}{2},x)}\}\) and \(0^0:=1=\lim_{x \to 0+}{x^x}.\) Then
\[\sup_{(x,y,z) \in D}{h(x,y,z)}= \begin{cases}
    4, & M \leq 1, \\
    \frac{(M^2+1)^2}{M^2}, & M \geq 1 
    \end{cases}.\]
\end{lemma}

\begin{proof}
For \(a+bx>0,\)
\[\frac{\partial}{\partial{x}}((a+bx)^{\pm{(a+bx)}})=\pm{(a+bx)^{\pm{(a+bx)}}} \cdot (b\log{(a+bx)}+b),\]
from which the partial derivatives of \(h\) satisfy
\[\frac{1}{h} \cdot \frac{\partial{h}}{\partial{y}}=\log{\frac{(z-2y)^2(x-y)}{(2y)^2(x+z-y)}},\]
\[\frac{1}{h} \cdot\frac{\partial{h}}{\partial{z}}=\log{\frac{(x-y+z)(1-x-z)}{(z-2y)(2-2x-z)}},\]
\[\frac{1}{h} \cdot \frac{\partial{h}}{\partial{x}}=2\log{M}+\log{\frac{(x-y+z)(1-x)(1-x-z)}{(x-y)(2-2x-z)^2}},\]
as the formula above gives
\[\frac{1}{h} \cdot \frac{\partial{h}}{\partial{y}}=-\log{(x+z-y)}-1+\log{(x-y)}+1-2(\log{2y}+1)+2(\log{(z-2y)}+1)\]
and similarly for \(\frac{\partial{h}}{\partial{z}}, \frac{\partial{h}}{\partial{x}}.\)
Moreover,
\[\frac{(z-2y)^2(x-y)}{(2y)^2(x+z-y)}=1+\frac{((z-2y)^2-4y^2)(x-y)-4zy^2}{(2y)^2(x+z-y)}=1+\frac{z^2x-(z^2+4xz)y}{(2y)^2(x+z-y)},\]
\[\frac{(x-y+z)(1-x-z)}{(z-2y)(2-2x-z)}=1+\frac{(x+3y)(1-x)-z(y+1)}{(z-2y)(2-2x-z)}\]
reveal the zeros of \(\frac{\partial{h}}{\partial{y}},\frac{\partial{h}}{\partial{z}},\) which will be subsequently employed.
\par
Let \((x_0,y_0,z_0) \in \arg{\sup{h(x,y,z)}}:\) then each coordinate either makes some constraining inequality equality or has its partial derivative zero. The analysis below uses this observation to find \((x_0,y_0,z_0)\) and thus \(\sup{h}\) (each case proves \(\sup{h} \leq f^2(M),\) and the third also entails \(\sup{h} \geq f^2(M)\)). For simplicity, drop the subscripts, and denote such a point by \((x,y,z):\) \(y,z,x\) are considered in this order.
\par
\underline{Case \(1:\)} \(y=0.\) 
\par
If \(z=0,\) then \(h(x,y,z)=M^{2x}2^{2-2x}\) yields \(\max{(4,M^2)}\leq \max{(4,\frac{(M^2+1)^2}{M^2}\chi_{M>1})}.\) 
\par
If \(z=1-x,\) then \(h(x,y,z)=M^{2x}\frac{1}{x^x(1-x)^{1-x}},\) whose supremum, attained at \(x=\frac{M^2}{M^2+1}\) as the derivative is \(h \cdot (2\log{M}-\log{\frac{x}{1-x}}),\) is \(\frac{1}{1-x}=M^2+1 \leq \max{(4,\frac{(M^2+1)^2}{M^2}\chi_{M>1})}.\)
\par
Else, \(\frac{\partial{h}}{\partial{z}}=0,\) from which \(z=x(1-x).\) Then \(h(x,y,z)=M^{2x}\frac{(2-x)^{2-x}}{(1-x)^{1-x}},\) whose critical point is given by \(\log{M^2}-\log{\frac{2-x}{1-x}}=0\) or \(x=\frac{M^2-2}{M^2-1}\) for \(M \geq \sqrt{2},\) yielding \(h(x,y,z)=\frac{(2-x)^2}{1-x}=\frac{M^4}{M^2-1}.\) The supremum in this case is 
\[\max{(4,M^2,\frac{M^4}{M^2-1}\chi_{M>\sqrt{2}})} \leq \max{(4,\frac{(M^2+1)^2}{M^2}\chi_{M>1})}:\] 
this is immediate if \(M^2 \leq 2,\) and if \(M^2>2,\) then \(\frac{M^4}{M^2-1}<\frac{(M^2+1)^2}{M^2}\) since for \(q>2,\) \[(q+1)^2(q-1)-q^3=q^2-q-1>0.\]
\par
\underline{Case \(2:\)} \(y=x.\) Then \(2x \leq z \leq 1-x,\) and \[h(x,y,z)=M^{2x}\frac{z^z}{(2x)^{2x}(z-2x)^{z-2x}} \cdot \frac{(2-2x-z)^{2-2x-z}}{(1-x)^{1-x}(1-x-z)^{1-x-z}}.\]
\par
If \(z=2x,\) then 
\[h(x,y,z)=M^{2x} \frac{(2-4x)^{2-4x}}{(1-x)^{1-x}(1-3x)^{1-3x}}.\]
Its critical points satisfy
\[\log{M^2}+\log{\frac{(1-x)(1-3x)^3}{(2-4x)^4}}=0,\]
giving \(h(x,y,z)=\frac{(2-4x)^2}{(1-x)(1-3x)}=M\sqrt{\frac{1-3x}{1-x}} \leq M<\max{(4,\frac{(M^2+1)^2}{M^2}\chi_{M>1})},\) while \(x \in \{0,\frac{1}{3}\}\) yields \(\max{(4,M^{2/3})} \leq \max{(4,M)} \leq \max{(4,\frac{(M^2+1)^2}{M^2}\chi_{M>1})}\) as \(\frac{(q^2+1)^2}{q^2}>q^2 \geq q\) for \(q \geq 1.\)
\par
If \(z=1-x,\) then
\[h(x,y,z)=M^{2x}\frac{(1-x)^{1-x}}{(2x)^{2x}(1-3x)^{1-3x}}.\]
At the endpoints, this function is \(1,M^{2/3},\) and its critical points solve
\[\log{M^2}+\log{\frac{(1-3x)^3}{(2x)^2(1-x)}}=0,\]
yielding
\[h(x,y,z)=\frac{1-x}{1-3x}=M^{2/3}(\frac{1-x}{2x})^{2/3}=4+\frac{11x-3}{1-3x} \leq \max{(4,M^{2/3}(4/3)^{2/3})} \leq \max{(4,\frac{(M^2+1)^2}{M^2}\chi_{M>1})}\]
using that either \(x<\frac{3}{11}\) or \(x \geq \frac{3}{11},\) and \(M^{2/3}(4/3)^{2/3}<M^2\) for \(M \geq 6.\)
\par
Else, \(z=\frac{4x(1-x)}{x+1}, x \leq 1/3,\) and
\[h(x,y,z)=M^{2x}\frac{1}{(2y)^{2y}(z-2y)^{-2y}} \cdot \frac{(2-2x-z)^{2-2x}}{(1-x)^{1-x}(1-x-z)^{1-x}}=M^{2x}\frac{(2-2x)^{2-2x}}{(1+x)^{1+x}(1-3x)^{1-3x}}.\]
At the endpoints, \(h\) is \(4,M^{2/3},\) and its critical points satisfy
\[\log{M^2}+\log{\frac{(1-3x)^3}{(2-2x)^2(1+x)}}=0,\]
giving
\[h(x,y,z)=\frac{(2-2x)^2}{(1-3x)(x+1)}=M^{2} \cdot \frac{(1-3x)^2}{(1+x)^2} \leq M^2  \leq \max{(4,\frac{(M^2+1)^2}{M^2}\chi_{M>1})}.\]
\par
\underline{Case \(3:\)} \(y=\frac{zx}{z+4x}.\) 
\par
Since \(0 \leq y \leq \min{(\frac{z}{2},x)}\) for all \(z,x \geq 0\) (\((x,y,z) \in D\) implies \(y=0\) if \(z=x=0\)), the constraints on \(z\) are \(0 \leq z \leq 1-x.\) If \(z=0,\) then \(z=y=0,\) a case already considered. If \(z=1-x,\) then \(y=\frac{x-x^2}{1+3x},\) and
\[h(x,y,z)=M^{2x}\frac{(x-y+z)^{x+z}}{(x-y)^{x}(z-2y)^{z}}=M^{2x}\frac{(1+x)^2}{(4x^2)^x(1-x^2)^{1-x}}=M^{2x}\frac{(1+x)^{1+x}}{(2x)^{2x}(1-x)^{1-x}}.\]
The critical point is given by
\[\log{M^2}+\log{\frac{(1-x^2)}{4x^2}}=0,\]
or \(x^2=\frac{M^2}{M^2+4},\) and is a global maximum with value
\[\frac{1+x}{1-x}=\frac{(M+\sqrt{M^2+4})^2}{4} \leq \max{(4,\frac{(M^2+1)^2}{M^2}\chi_{M>1})}:\] 
if \(M^2 \leq 3/2,\) then \(\frac{(M+\sqrt{M^2+4})^2}{4} \leq 4;\) else, \(M^2 \geq 3/2,\) and \(\frac{(M+\sqrt{M^2+4})^2}{4}< \frac{(M^2+1)^2}{M^2}\) as \(M\sqrt{M^2+4}<M^2+2.\)
\par
Lastly, \(z=\frac{(x+3y)(1-x)}{y+1}\) and so \(y=\frac{z-x+x^2}{3-3x-z},\) whereby \(z^2(x+1)+z \cdot 4x^2+4x^3-4x^2=0.\) The roots of this quadratic equation are \(-2x, \frac{2x-2x^2}{x+1}.\) Hence \(z=\frac{2x-2x^2}{x+1}\) and \(y=\frac{x-x^2}{x+3}\) (\((x,y,z) \in D\) for \(0 \leq x \leq 1\)).
\par
If \(x \in \{0,1\},\) then \(y=z=0,\) a situation already analyzed. Else, \(x\) is a critical point,
\[2\log{M}+\log{\frac{(x-y+z)(1-x)(1-x-z)}{(x-y)(2-2x-z)^2}}=0\]
\[2\log{M}+\log{[(1+\frac{z}{x-y}) \cdot \frac{1-\frac{z}{1-x}}{(2-\frac{z}{1-x})^2}]}=0,\]
with \(\frac{z}{x-y}=\frac{(x+3)(1-x)}{(1+x)^2}, \frac{z}{1-x}=\frac{2x}{x+1},\) from which
\[2\log{M}+\log{\frac{4}{(x+1)^2} \cdot \frac{1-x^2}{4}}=0,\]
\[2\log{M}+\log{\frac{1-x}{1+x}}=0,\]
or \(x=\frac{M^2-1}{M^2+1}\) for \(M>1.\)
Since all the partial derivatives are zero,
\[h(x,y,z)=\frac{(2-2x-z)^2}{(1-x)(1-x-z)}=\frac{4}{1-x^2}=\frac{(M^2+1)^2}{M^2}.\]
\end{proof}

\section{Limiting Distributions of \(||A||\) and \(\lambda_1(A)\)}\label{section3}

This last section completes the proof of Theorem~\ref{theorem1} by concluding (\ref{desiredlim}) (subsections~\ref{uppbound} and \ref{reverseineq} treat (\ref{firstfinalgoal}) and (\ref{secondfinalgoal}), respectively) and its analog for \(||A||\) replaced by \(\lambda_1(A)\) (subsection~\ref{largesteigenvalue}).

\subsection{The Upper Bound}\label{uppbound}

This subsection justifies (\ref{firstfinalgoal}). Two tools are used towards this: the bound on the conditional expectation \(*\) given by (\ref{trace}), and the following linear algebra result, which provides a connection between \(||A_\kappa||,||A_s||\) and \(tr((A_s+A_{B,\kappa})^{2p})-tr(A_s^{2p}).\) For symmetric matrices, denote by \(\lambda_i(\cdot)\) the \(i^{th}\) largest eigenvalue for \(i \geq 1,\) and by convention, \(\lambda_i(Q)\) implicitly assumes \(i \geq 1, \lambda_i(Q)=0, i>n\) when \(Q \in \mathbb{R}^{n \times n}.\) 

\begin{lemma}\label{linalglemma}
Suppose \(S, Q \in \mathbb{R}^{n \times n}\) are symmetric matrices with \(\lambda_1(S) \geq 0, \lambda_{2m+1}(Q)=0\) for some integer \(m \in [1,\frac{n}{6}-1].\) Then for \(p \in \mathbb{N},\)
\[||S+Q||^{2p}-7m \cdot ||S||^{2p} \leq tr((S+Q)^{2p})-tr(S^{2p}) \leq 4m \cdot ||S+Q||^{2p}.\]
\end{lemma}

\begin{proof}
Since \(S, Q\) are symmetric, for \(\tilde{\sigma}_i=\lambda_i(S+Q), \sigma_i=\lambda_i(S), 1 \leq i \leq n,\)
\[tr((S+Q)^{2p})-tr(S^{2p})=\sum_{1 \leq i \leq n}{\tilde{\sigma}^{2p}_i}-\sum_{1 \leq i \leq n}{\sigma^{2p}_i}.\]
Weyl's inequalities,
\[\lambda_{k}(S)+\lambda_{n+j-k}(Q) \leq \lambda_{j}(S+Q) \leq \lambda_{l}(S)+\lambda_{j+1-l}(Q),\]
give when \(j \geq 2m+1,\)
\[\lambda_{j}(S)+\lambda_{n}(Q)=\sigma_j \leq \tilde{\sigma}_j \leq \sigma_{j-2m}=\lambda_{j-2m}(S)+\lambda_{2m+1}(Q),\]
the first inequality holding for all \(1 \leq j \leq n.\) Let \(t \geq 1\) be minimal with \(\sigma_t \geq 0 \geq \sigma_{t+1}.\) 
\par
The ensuing case-by-case analysis yields the statement of the lemma (\(1, 2, 3\) cover the lower bound, and \(4,5\) the upper bound).
\begin{enumerate}
    \item \(t \leq 2m:\) 
    \[tr((S+Q)^{2p})-tr(S^{2p})=\sum_{i \leq n}{\tilde{\sigma}_i^{2p}}-\sum_{i \leq n}{\sigma_i^{2p}} \geq\]
    \[\geq \tilde{\sigma}_1^{2p}+\sum_{1<i \leq t}{\sigma_i^{2p}} +\sum_{t+1 \leq i \leq t+2m}{\tilde{\sigma}_{i}^{2p}}+\sum_{t+2m< i \leq n-2m}{\sigma_{i-2m}^{2p}}+\sum_{n-2m< i \leq n}{\tilde{\sigma}_{i}^{2p}}-\sum_{i \leq n}{\sigma_i^{2p}} \geq\]
    \[\geq \tilde{\sigma}_1^{2p}+\sum_{n-2m< i \leq n}{\tilde{\sigma}_{i}^{2p}}-\sigma_1^{2p}-
    \sum_{n-4m< i \leq n}{\sigma_i^{2p}} \geq \max_{1 \leq i \leq n}{\tilde{\sigma}_i^{2p}}-(4m+1)\max_{1 \leq i \leq n}{\sigma_i^{2p}};\]

    \item \(2m<t<n-4m:\)  
    \[tr((S+Q)^{2p})-tr(S^{2p})=\sum_{i \leq n}{\tilde{\sigma}_i^{2p}}-\sum_{i \leq n}{\sigma_i^{2p}} \geq\]
    \[\geq \sum_{i \leq 2m}{\tilde{\sigma}_i^{2p}}+\sum_{2m<i \leq t }{\sigma_i^{2p}}+\sum_{t< i \leq t+2m}{\tilde{\sigma}_{i}^{2p}}+\sum_{t+2m< i \leq n-2m}{\sigma_{i-2m}^{2p}}+\sum_{n-2m< i \leq n}{\tilde{\sigma}_{i}^{2p}}-\sum_{i \leq n}{\sigma_i^{2p}} \geq\]
    \[\geq \sum_{i \leq 2m}{\tilde{\sigma}_i^{2p}}+\sum_{n-2m< i \leq n}{\tilde{\sigma}_{i}^{2p}}-\sum_{i \leq 2m}{\sigma_i^{2p}}-\sum_{n-4m< i \leq n}{\sigma_i^{2p}} \geq \max_{1 \leq i \leq n}{\tilde{\sigma}_i^{2p}}-6m\max_{1 \leq i \leq n}{\sigma_i^{2p}};\]
    
    \item \(n-4m \leq t:\) 
    \[tr((S+Q)^{2p})-tr(S^{2p})=\sum_{i \leq n}{\tilde{\sigma}_i^{2p}}-\sum_{i \leq n}{\sigma_i^{2p}} \geq\]
    \[\geq \sum_{i \leq 2m}{\tilde{\sigma}_i^{2p}}+\sum_{2m<i<t }{\sigma_i^{2p}}+\sum_{t \leq i \leq n}{\tilde{\sigma}_{i}^{2p}}-\sum_{i \leq n}{\sigma_i^{2p}} \geq\]
    \[\geq \sum_{i \leq 2m}{\tilde{\sigma}_i^{2p}}+\tilde{\sigma}_{n}^{2p}-\sum_{i \leq 2m}{\sigma_i^{2p}}-\sum_{t \leq i \leq n}{\sigma_i^{2p}} \geq \max_{1 \leq i \leq n}{\tilde{\sigma}_i^{2p}}-(6m+1)\max_{1 \leq i \leq n}{\sigma_i^{2p}};\]

    \item \(t \leq 2m:\)
     \[tr((S+Q)^{2p})-tr(S^{2p})=\sum_{i \leq n}{\tilde{\sigma}_i^{2p}}-\sum_{i \leq n}{\sigma_i^{2p}} \leq \sum_{i \leq 2m+t}{\tilde{\sigma}_i^{2p}}+\sum_{2m+t<i \leq n}{\sigma_{i}^{2p}}-\sum_{i \leq n}{\sigma_i^{2p}} \leq 4m\max_{1 \leq i \leq n}{\tilde{\sigma}_i^{2p}};\]
    
    \item \(t>2m:\)
    \[tr((S+Q)^{2p})-tr(S^{2p})=\sum_{i \leq n}{\tilde{\sigma}_i^{2p}}-\sum_{i \leq n}{\sigma_i^{2p}} \leq \sum_{i \leq 2m}{\tilde{\sigma}_i^{2p}}+\sum_{2m<i \leq t}{\sigma_{i-2m}^{2p}}+\sum_{t<i \leq \min(t+2m,n)}{\tilde{\sigma}_{i}^{2p}}+\]
    \[+\sum_{\min(t+2m,n)<i \leq n}{\sigma_{i}^{2p}}-\sum_{i \leq n}{\sigma_i^{2p}} \leq 4m\max_{1 \leq i \leq n}{\tilde{\sigma}_i^{2p}}.\]
\end{enumerate}

\end{proof}

(\ref{firstfinalgoal}) can now be concluded. Conditional on \(E(S,\kappa,M),\) Lemma~\ref{linalglemma} implies for \(n\) sufficiently large,
\[tr((A_s+A_{B,\kappa})^{2p})-tr(A_s^{2p}) \geq ||A_\kappa||^{2p}-7m \cdot ||A_s||^{2p},\]
and so
\[(tr((A_s+A_{B,\kappa})^{2p})-tr(A_s^{2p}))_{-} \leq 7m \cdot ||A_s||^{2p}.\]
Since \(\mathbb{P}(X>a) \leq \mathbb{E}[X^p+X_{-}^p]a^{-p}\) for any real-valued random variable \(X\) and \(a>0,\)
\[\mathbb{P}_*(||A_\kappa||>f(M)+\epsilon) \leq (f(M)+\epsilon)^{-2p}(\mathbb{E}_{*}[tr((A_s+A_{B,\kappa})^{2p})-tr(A_s^{2p})]+7m \cdot \mathbb{E}_{*}[||A_s||^{2p}]).\]
\par
Arguing as in subsection (\ref{twistedmethod}), 
\[\mathbb{E}_{*}[||A_s||^{2p}] \leq c(\kappa,c)\mathbb{E}[||A_s||^{2p}],\]
and Theorem~\ref{thbenaychpeche} for \(\mu=1, \gamma'=\frac{1}{2}, \gamma=\frac{1}{4}-\delta_1, \gamma''=\frac{\delta_1}{2}, s=\lfloor{} n^{\gamma''} \rfloor \geq  n^{\gamma''}/2\) yields
\[\mathbb{E}[||A_s||^{2p}] \leq (2+\epsilon)^{2p}+\mathbb{E}[||A_s||^{2s}] \cdot (2+\epsilon)^{2p-2s} \leq (2+\epsilon)^{2p}+(2+\epsilon)^{2p} \cdot n^3 (2+\epsilon)^{-n^{\gamma''}},\]
with the first term dominating the second for \(n\) large enough. Consequently,
\[\mathbb{P}_*(||A_\kappa||>f(M)+\epsilon) \leq (f(M)+\epsilon)^{-2p}(\mathbb{E}_{*}[tr((A_s+A_{B,\kappa})^{2p})-tr(A_s^{2p})]+14m(2+\epsilon)^{2p}),\]
whereby (\ref{trace}) for \(\kappa=\epsilon\) yields
\[\mathbb{P}_*(||A_\kappa||>f(M)+\epsilon) \leq 2mc(\epsilon,c)(1+c_1(M,\epsilon))^{-2p}+n^{-\delta/2}c(p)+14m(1+c_3(M,\epsilon))^{-2p}\]
for some \(c_i(M,\epsilon)>0,\) and \(n,p\) sufficiently large. Given the growth hierarchy \(m, p, n\) form, this last inequality entails (\ref{firstfinalgoal}).

\subsection{The Lower Bound}\label{reverseineq}

This subsection proves (\ref{secondfinalgoal}) by justifying for \(n \geq n(\delta, \kappa, M),\)
\begin{equation}\label{condexp}
    \mathbb{E}_{**}[tr(A^{2p}_\kappa)-tr(A^{2p}_s)] \geq \frac{c_0(\kappa,c)}{p}(1-2mn^{-1})^{p}(1-n^{-2\delta})^ps(\max{A},p),
\end{equation}
\begin{equation}\label{condvar}
    Var_{**}(tr(A^{2p}_\kappa)-tr(A^{2p}_s)) \leq n^{-1/2}[(\max{(M,2)})^{4p-2}(2m)^{4p}(4p)^{64p^2}+s(2p-1,M)].
\end{equation}
\par
Begin with (\ref{condexp}). The key observation is that anew solely even cycles contribute (in particular, the considered expectation is a sum of non-negative terms):

\begin{lemma}\label{lemma6}
Suppose a cycle \(\mathbf{i}\) contains some edge belonging to \(A_{s}\) and is not even (its length might be odd). Then there exists an undirected edge belonging to \(A_s\) appearing an odd number of times in \(\mathbf{i}.\)
\end{lemma}

\begin{proof}
Compress the clusters of edges belonging to \(A_{B,\kappa}\) to points or single edges (as in Lemma~\ref{goodcycles}), and note this procedure leaves the edges of \(\mathbf{i}\) belonging to \(A_s\) intact and does not change the parity of the cycle: i.e., this new cycle \(\mathbf{i}_c \ne \emptyset\) is not even and shares with \(\mathbf{i}\) its edges belonging to \(A_s\) (including which ones are marked). If \(\mathbf{i}_c\) has no edge belonging to \(A_{B,\kappa},\) then the conclusion follows. Else, there is an undirected edge \(vw\) belonging to \(A_{B,\kappa}\) appearing an odd number of times in \(\mathbf{i}_c:\) let the indices of these edges be \(1 \leq p_1<p_2<...<p_{2k+1} \leq p',\) where \(p'\) is the length of \(\mathbf{i}_c.\) A case-by-case analysis and \(u \ne v\) show either \(u\) or \(v\) is adjacent to an odd number of edges belonging to \(A_{s}\) (let \(\mathcal{P}_{odd}\) be the set of vertices of \(\mathbf{i}_c\) having this property), whereby the claim of the lemma ensues:
\begin{enumerate} 
    \item \(p_1>1, p_{2k+1}<p':\) \(v, w \in \mathcal{P}_{odd};\) 
    
    \item \(p_1=1, p_{2k+1}<p':\) \(i_1 \in \mathcal{P}_{odd};\)
    
    \item \(p_1>1, p_{2k+1}=p':\) \(i_{p'-1} \in \mathcal{P}_{odd};\)
    
    \item \(p_1=1, p_{2k+1}=p':\) \(i_1 \in \mathcal{P}_{odd}\) (as \(i_{p'-1}=i_1\)).
\end{enumerate}
\end{proof}

Since the left-hand side of (\ref{condexp}) is a sum of non-negative terms, the inequality follows from the description of cycles of type \((III)\) in subsection~\ref{twistedmethod} (see proof of (\ref{typeIII})) and \(S \ne \emptyset:\) all cycles of type \((III)\) with \(|a_{uv}|=\max_{1 \leq i \leq j \leq n}{|a_{ij}|}\) have expectation at least \(c_0(\kappa,c)(\max{A})^{2l}(n^{-1}\mathbb{E}[a^2_{11}\chi_{|a_{11}|\leq n^{\delta}}])^{p-l},\) where \(2l\) is the multiplicity of \(uv\) in the cycle, and there are at least \((n-2m)^{p-l}\) possibilities for choosing the remaining vertices of \(\mathbf{i}\) because restricting them to \(\{1,2, \hspace{0.05cm} ... \hspace{0.05cm}, n\} - \{t: \exists s, (\min{(s,t)},\max{(s,t)}) \in S\},\) a set of size at least \(n-2m,\) ensures no edge belonging to \(A_{B,\kappa}\) is created by any such assignment.
\par
Consider now (\ref{condvar}). Clearly,
\[Var_{**}[(tr(A^{2p}_\kappa)-tr(A^{2p}_s)]=\sum_{(\mathbf{i},\mathbf{j})}{(\mathbb{E}_{**}[a_{\mathbf{i}} \cdot a_{\mathbf{j}}]-\mathbb{E}_{**}[a_{\mathbf{i}}] \cdot \mathbb{E}_{**}[a_{\mathbf{j}}])},\]
where \(\mathbf{i},\mathbf{j}\) are cycles of length \(2p\) containing at least one edge belonging to \(A_{B,\kappa}.\)  Proceed in the same vein as Sinai and Soshnikov~[\ref{sinaisosh}] did when analyzing the variance of large moments of the trace of a Wigner matrix. By independence, the contribution of \((\mathbf{i},\mathbf{j})\) is non-zero iff \(\mathbf{i}\) and \(\mathbf{j}\) share at least one undirected edge belonging to \(A_{s},\) and every undirected edge in their union \(\mathbf{i} \cup \mathbf{j}\) appears an even number of times (if they share no edge, then they are independent; else, if there is an edge in the union appearing an odd number of times, then both terms are zero by Lemma~\ref{lemma6} and symmetry). 
\par
A crucial step in [\ref{sinaisosh}] is mapping such pairs \((\mathbf{i},\mathbf{j})\) to even cycles \(\mathcal{P}\) of length \(2 \cdot 2p-2=4p-2.\) Let \(i_{t-1}i_{t}=j_{s-1}j_{s}\) with \(t,s\) minimal in this order (i.e., \(t=\min{\{1 \leq k \leq 2p, \exists 1 \leq q \leq 2p, i_{k-1}i_{k}=j_{q-1}j_{q}\}}, s=\min{\{1 \leq q \leq 2p, j_{q-1}j_{q}=i_{t-1}i_{t}\}},\) where only edges belonging to \(A_s\) are considered). Then \(\mathcal{P}\) is 
obtained by gluing these two cycles along this common edge, which then gets erased. Put differently, \(\mathcal{P}\) traverses \(\mathbf{i}\) up to \(i_{t-1}i_{t},\) which is then used as a bridge to switch to \(\mathbf{j},\) traverse all of it, and get back to the rest of \(\mathbf{i}\) upon returning to \(j_{s-1}j_s=i_{t-1}i_{t}.\) More specifically, if \((i_{t-1},i_{t})=(j_{s-1},j_{s}),\) then
\[\mathcal{P}:=(i_0, \hspace{0.05cm} ... \hspace{0.05cm},i_{t-1},j_{s-2}, \hspace{0.05cm} ... \hspace{0.05cm},j_0,j_{2p-1}, \hspace{0.05cm} ... \hspace{0.05cm}, j_{s},i_{t+1}, \hspace{0.05cm} ... \hspace{0.05cm},i_{2p});\]
else, \((i_{t-1},i_{t})=(j_{s},j_{s-1}),\) and 
\[\mathcal{P}:=(i_0, \hspace{0.05cm} ... \hspace{0.05cm}, i_{t-1},j_{s+1}, \hspace{0.05cm} ... \hspace{0.05cm}, j_{2p-1},j_0, \hspace{0.05cm} ... \hspace{0.05cm}, j_{s-1},i_{t+1}, \hspace{0.05cm} ... \hspace{0.05cm}, i_{2p}).\]
Evidently, \(\mathcal{P}\) is an even cycle of length \(2 \cdot 2p-2=4p-2.\) 
\par
Since for the conditional expectation \(\mathbb{E}_{**}[\cdot]\) a similar split to the one in (\ref{sub}) occurs,
\[|\mathbb{E}_{**}[a_{\mathbf{i}} \cdot a_{\mathbf{j}}]-\mathbb{E}_{**}[a_{\mathbf{i}}] \cdot \mathbb{E}_{**}[a_{\mathbf{j}}]| \leq |\mathbb{E}_{**}[a_{\mathbf{i}} \cdot a_{\mathbf{j}}]|.\]
Hence
\[Var_{**}[(tr(A^{2p}_\kappa)-tr(A^{2p}_s)] \leq 4p \cdot n^{-2p} \sum_{\mathcal{P}, 0 \leq r \leq 2p-1}{\mathbb{E}_{**}[a_{\mathcal{P}}a_{q_rq_{r+2p-1}}^2]},\]
where \(\mathcal{P}=(q_0,q_1, \hspace{0.05cm} ... \hspace{0.05cm},q_{4p-3},q_0)\) is an even cycle with at least two edges belonging to \(A_{B,\kappa}\) inasmuch as for any such \(\mathcal{P}\) and \(0 \leq r \leq 2p-1\) there are at most \(4p\) pairs \((\mathbf{i},\mathbf{j})\) mapped to it (\(i_{t-1}=q_r, i_t=q_{t-1+2p-1},\) and it remains to choose whether \((i_{t-1},i_t)=(j_{s-1},j_s)\) or \((i_{t-1},i_t)=(j_{s},j_{s-1}),\) and the first vertex of \(\mathbf{j},\) which can be done in at most \(2 \cdot 2p=4p\) ways). Because \(a_{q_rq_{r+2p-1}}\) belongs to \(A_s,\) 
\[\mathbb{E}_{**}[a_{\mathcal{P}}a_{q_rq_{r+2p-1}}^2] \leq
n^{2(1/4-\delta_1)}\mathbb{E}_{**}[a_{\mathcal{P}}],\]
and reasoning as in subsection~\ref{twistedmethod},
\[\sum_{\mathcal{P}}{\mathbb{E}_{**}[a_{\mathcal{P}}]} \leq c(\kappa,c) \cdot n^{2p-1} [(4m-2) \cdot M^{4p-2}+4mn^{-\delta}e^{16}\sum_{1 \leq l \leq 2p-1}{\binom{4p-2}{2l} (2m)^{2l} \cdot 2^{4p-2-2l} ((2l+2)!)^{4l}  M^{2l}}+s(2p-1,M)]\]
yielding the conditional variance is at most
\[(4p \cdot n^{-2p}) \cdot n^{2(1/4-\delta_1)} \cdot c(\kappa,c)n^{2p-1} \cdot [(\max{(M,2)})^{4p-2}(2m)^{4p}(4p)^{64p^2}+s(2p-1,M)].\]
\par
Now (\ref{secondfinalgoal}) can be concluded. Lemma~\ref{linalglemma} gives, conditional on \(E(S,M,\kappa),\)
\[||A_\kappa||^{2p} \geq \frac{1}{4m}(tr(A^{2p}_\kappa)-tr(A^{2p}_s)).\]
Since \(2 \leq f(\max{A}) \leq f(M),\) (\ref{condexp}) yields for \(\epsilon \in (0,1)\)
\[\frac{1}{4m}\mathbb{E}_{**}[(tr(A^{2p}_\kappa)-tr(A^{2p}_s)] \geq \frac{c_0(\kappa,c)}{4m}(1-2n^{-2\delta})^{p}(f(\max{A})-\epsilon/2)^{2p} \geq 2(f(\max{A})-\epsilon)^{2p}\]
(subsection~\ref{closedformcombfunction} entails \(s(p,\cdot)^{1/2p} \to f(\cdot)\) uniformly on compact subsets of \((0,\infty);\) thus, for all \(p \geq p(M,\epsilon,\kappa),\) \(s(\max{A},p) \geq (f(\max{A})-\epsilon/2)^{2p}\)). Chebyshev's inequality gives
\[\mathbb{P}_{**}(||A_\kappa||<f(\max{A})-\epsilon)=\mathbb{P}_{**}(||A_\kappa||^{2p}<(f(\max{A})-\epsilon)^{2p}) \leq \frac{Var_{**}[(tr(A^{2p}_\kappa)-tr(A^{2p}_s)]}{16m^2(f(\max{A})-\epsilon)^{2p}}=o(1)\]
using (\ref{condvar}).

\subsection{The Largest Eigenvalue}\label{largesteigenvalue}

This subsection completes the proof of Theorem~\ref{theorem1} by arguing (\ref{desiredlim}) remains true when \(||A||\) is replaced by \(\lambda_1(A).\) The first inequality is immediate from (\ref{desiredlim}), while for the second, in the same spirit as before, it suffices to show for \(\kappa=\delta>0\) and \(0<\epsilon<\frac{f(1+\delta)-2}{8},\)
\[\lim_{n \to \infty}{\mathbb{P}_{**}(\lambda_1(A_\kappa) \leq f(\max{A})-\epsilon,\max{A} \geq 1+\delta)}=0\]
(if \(\max{A} \leq 1,\) then (\ref{liminf}) implies the desired result). Consider the following modified version of Lemma~\ref{linalglemma}:

\begin{lemma}\label{lemma8}
Suppose \(S, Q \in \mathbb{R}^{n \times n}\) are symmetric matrices with \(\lambda_{2m+1}(Q)=0\) for some integer \(m \in [1,\frac{n}{4}-1].\) Then for \(p \in \mathbb{N},\)
\[tr((S+Q)^{2p+1})-tr(S^{2p+1}) \leq 2m \cdot (\lambda_1(S+Q))^{2p+1}+(\lambda_n(S+Q))^{2p+1}+3m \cdot ||S||^{2p+1}.\]
\end{lemma}

\begin{proof}
Keeping the notation from the proof of Lemma~\ref{linalglemma}, 
\[tr((S+Q)^{2p+1})-tr(S^{2p+1})=\sum_{i \leq n}{\tilde{\sigma}_i^{2p+1}}-\sum_{i \leq n}{\sigma_i^{2p+1}} \leq \sum_{i \leq 2m}{\tilde{\sigma}_1^{2p+1}}+\sum_{2m<i \leq n-1}{\sigma_{i-2m}^{2p+1}}+\tilde{\sigma}_n^{2p+1}-\sum_{i \leq n}{\sigma_i^{2p+1}} \leq\]
\[\leq 2m \cdot \tilde{\sigma}_1^{2p+1}+\tilde{\sigma}_n^{2p+1}+(2m+1) \cdot (\max_{1 \leq i \leq n}{|\sigma_i|})^{2p+1}.\]
\end{proof}

Conditional on \(**,\) Lemma~\ref{lemma8} gives
\[tr(A_\kappa^{2p+1})-tr(A_s^{2p+1}) \leq 2m \cdot (\lambda_1(A_\kappa))^{2p+1}+(\lambda_n(A_\kappa))^{2p+1}+3m \cdot ||A_s||^{2p+1};\]
if additionally \(\lambda_1(A_\kappa) \leq f(\max{A})-\epsilon<f(\max{A})-\epsilon/2 \leq ||A_{\kappa}||, ||A_s|| \leq 2+\epsilon, \max{A} \geq 1+\delta,\) then
\[tr(A_\kappa^{2p+1})-tr(A_s^{2p+1}) \leq 2m \cdot (f(\max{A})-\epsilon)^{2p+1}-(f(\max{A})-\epsilon/2)^{2p+1}+3m \cdot (2+\epsilon)^{2p+1} \leq -\frac{(f(\max{A})-\epsilon/2)^{2p+1}}{2}.\]
Therefore,
\[\mathbb{P}_{**}(\lambda_1(A_\kappa) \leq f(\max{A})-\epsilon, \max{A} \geq 1+\delta) \leq \mathbb{P}_{**}(||A_\kappa||< f(\max{A})-\epsilon/2)+\mathbb{P}_{**}(||A_s||>2+\epsilon)+\]
\[+\frac{4Var_{**}(tr(A_\kappa^{2p+1})-tr(A_s^{2p+1}))}{(f(\max{A})-\epsilon/2)^{4p+2}}=o(1)\]
since Lemma~\ref{lemma6} yields
\[\mathbb{E}_{**}[tr(A_\kappa^{2p+1})-tr(A_s^{2p+1})]=0\]
(a cycle of odd length is not even and contains some edge belonging to \(A_s\)), and reasoning as for (\ref{condvar}), 
\[Var_{**}(tr(A_\kappa^{2p+1})-tr(A_s^{2p+1})) \leq n^{-1/2}[(\max{(M,2)})^{4p-2}(2m)^{4p}(4p)^{64p^2}+s(2p,M)].\]
This completes the proof of Theorem~\ref{theorem1}.
\par
Regarding the largest \(k\) eigenvalues of \(A\) for \(k \in \mathbb{N}\) fixed, a similar rationale to \(k=1\) could be used, although the combinatorics would be more involved. Denote by \((\lambda_{(i)}(\frac{1}{\sqrt{n}} A))_{1 \leq i \leq n},(max_{(l)}(A))_{1 \leq l \leq \frac{n^2+n}{2}}\) the ordered statistics of \((|\lambda_i(\frac{1}{\sqrt{n}}A)|)_{1 \leq i \leq n}, (\frac{1}{\sqrt{n}}|a_{ij}|)_{1 \leq i \leq j \leq n},\) respectively with \(\lambda_{(1)}(\frac{1}{\sqrt{n}} A)=\frac{1}{\sqrt{n}}||A||,\) and \(max_{(1)}(A)=max(A).\) Use induction on \(k\) to show
\[\lambda_{k}(\frac{1}{\sqrt{n}}A)-f(max_{(k)}(A)) \xrightarrow[]{p} 0,\]
which in conjunction with symmetry would imply
\[\lambda_{n+1-k}(\frac{1}{\sqrt{n}}A)+f(max_{(k)}(A)) \xrightarrow[]{p} 0,\]
whereby
\[\lambda_{(k)}(\frac{1}{\sqrt{n}}A)-f(max_{(k)}(A)) \xrightarrow[]{p} 0.\]
The base case is Theorem~\ref{theorem1}; suppose the result holds for \(k \geq 1,\) and consider next \(k+1.\) Similarly to the case \(k=1,\) prove first   \[\lambda_{(k+1)}(\frac{1}{\sqrt{n}}A)-f(max_{(k+1)}(A)) \xrightarrow[]{p} 0,\]
and second justify this holds also for \(\lambda_{k+1}(\frac{1}{\sqrt{n}}A).\) Since the behavior of the largest \(k\) eigenvalues is known, consider
\begin{equation}\label{symsum}
    \sum_{i_1<i_2<...<i_{k+1}}{\lambda^{2p}_{i_1}(\frac{1}{\sqrt{n}}A) \cdot \lambda^{2p}_{i_2}(\frac{1}{\sqrt{n}}A) \cdot ... \cdot \lambda^{2p}_{i_{k+1}}(\frac{1}{\sqrt{n}}A)}
\end{equation}
for \(p \in \mathbb{N}.\) The dominant term is 
\[\lambda^{2p}_{(1)}(\frac{1}{\sqrt{n}}A) \cdot \lambda^{2p}_{(2)}(\frac{1}{\sqrt{n}}A) \cdot ... \cdot \lambda^{2p}_{(k+1)}(\frac{1}{\sqrt{n}}A),\]
for which the induction hypothesis gives it is roughly 
\[f^{2p}(max_{(1)}(A)) \cdot f^{2p}(max_{(2)}(A)) \cdot ... \cdot f^{2p}(max_{(k)}(A)) \cdot \lambda^{2p}_{(k+1)}(\frac{1}{\sqrt{n}}A);\]
(\ref{symsum}) could be expressed using
\[tr(\frac{1}{\sqrt{n}}A)^{2q}=\lambda^{2q}_1(\frac{1}{\sqrt{n}}A)+\lambda^{2q}_2(\frac{1}{\sqrt{n}}A)+...+\lambda^{2q}_n(\frac{1}{\sqrt{n}}A)\]
for \(q \in \{p,2p, \hspace{0.05cm} ... \hspace{0.05cm} ,(k+1)p\},\) by employing trace difference \(tr(A_{\kappa}^{2q})-tr(A_s^{2q})\) instead (to ensure they capture just the edge eigenvalues of \(A\)), a rationale as in section~\ref{section3} and the analysis from section~\ref{section2} would give (\ref{symsum}) is in (conditional) expectation
\[f^{2p}(max_{(1)}(A)) \cdot f^{2p}(max_{(2)}(A)) \cdot ... \cdot f^{2p}(max_{(k+1)}(A))(1+o(1))\] 
and its variance small, whereby the desired result for \(\lambda_{(k+1)}(\frac{1}{\sqrt{n}}A)\) ensues; finally,
\begin{equation}\label{symsum2}
    \sum_{\{i_1,i_2, \hspace{0.05cm} ... \hspace{0.05cm},i_{k+1}\}}{\lambda^{2p}_{i_1}(\frac{1}{\sqrt{n}}A) \cdot \lambda^{2p}_{i_2}(\frac{1}{\sqrt{n}}A) \cdot ... \cdot \lambda^{2p}_{i_{k}}(\frac{1}{\sqrt{n}}A) \cdot \lambda^{2p+1}_{i_{k+1}}(\frac{1}{\sqrt{n}}A)},
\end{equation}
where the \(k+1\) indices are pairwise distinct (i.e., they are the elements of a set), would yield \(\lambda_{k}(\frac{1}{\sqrt{n}}A)<\lambda_{(k)}(\frac{1}{\sqrt{n}}A)-\epsilon\) occurs with small probability (else, (\ref{symsum2}) would be negative because  \(\lambda_{n-k}(\frac{1}{\sqrt{n}}A)=-\lambda_{(k+1)}(\frac{1}{\sqrt{n}}A)\) given the induction hypothesis, which should occur with small probability since its expectation would be \(0\) and its variance small).

\vspace{0.3cm}
\textbf{Acknowledgements:} The author would like to thank professors George Papanicolaou and Lenya Ryzhik for their feedback on the expository aspects of this paper.

\addcontentsline{toc}{section}{References}
\section*{References}

\begin{enumerate}

    \item\label{auffinger} A. Auffinger, G. Ben-Arous, and S. Péché, \textit{Poisson convergence for the largest eigenvalues of heavy tailed random matrices}, Ann. Inst. H. Poincaré Probab. Statist., Vol. \(45,\) No. \(3, 589-610, 2009.\) 
    
    \item\label{baisilv} Z. D. Bai, and J. Silverstein, \textit{Spectral Analysis of Large Dimensional Random Matrices}, Springer Series in Mathematics, Second Edition, \(2010.\)
    
     \item\label{baietal} Z. D. Bai, J. Silverstein, and Y. Q. Yin, \textit{A note on the largest eigenvalue of a large dimensional sample covariance matrix},  J. Multivariate Anal., \(26, 166-168, 1988.\)
     
     \item\label{baiyin2} Z. D. Bai, and Y. Q. Yin. \textit{Limit of the Smallest Eigenvalue of a Large Dimensional Sample Covariance Matrix}, Ann. of Probab., Vol. \(21,\) No. \(3, 1275-1294, 1993.\)
    
    \item\label{baiyin} Z. D. Bai, and Y. Q. Yin, \textit{Necessary and Sufficient Conditions for Almost Sure Convergence of the Largest Eigenvalue of a Wigner Matrix}, Ann. of Probab., Vol. \(16,\) No. \(4, 1729-1741, 1988.\)
    
    \item\label{benaychpeche} F. Benaych-Georges, and S. Péché, \textit{Localization and Delocalization for Band Matrices},
    Ann. Inst. H. Poincaré Probab. Statist., Vol. \(50,\) No. \(4, 1385-1403, 2014.\) 

    \item\label{khorunzhiy} O. Khorunzhiy, \textit{High Moments of Large Wigner Random Matrices and Asymptotic Properties of the Spectral Norm}, Random Operators and Stochastic Equations, \(20, 25-68, 2012.\)
    
    \item\label{leadbetteretal} M. R. Leadbetter, G. Lindgren, and H. Rootzén, \textit{Extremes and Related Properties of Random Sequences and Processes}, Springer-Verlag, New York, \(1983.\)
    
    \item\label{leeyin} J. O. Lee, and J. Yin, \textit{A Necessary and Sufficient Condition for Edge Universality of Wigner Matrices}, \(2012,\) Duke Math. J., Vol. \(163,\) No. \(1, 117-173, 2014.\)
    
    \item\label{ruzmaikina} A. Ruzmaikina, \textit{Universality of the Edge Distribution of Eigenvalues of Wigner Random Matrices with Polynomially Decaying Distributions of Entries}, Commun. Math. Phys., Vol. \(261,\) Issue \(2, 277-296,\) \(2006.\)

    \item\label{sinaisosh} Ya. Sinai, and A. Soshnikov, \textit{Central Limit Theorem for Traces of Large Random Symmetric Matrices With Independent Matrix Elements}, Bol. Soc. Brasil. Mat., Vol. \(29,\) No. \(1, 1-24, 1998.\)
    
    \item\label{sinaisosh2} Ya. Sinai, and A. Soshnikov, \textit{A Refinement of Wigner's Semicircle Law in a Neighborhood of the Spectrum Edge for Random Symmetric Matrices}, Functional Analysis and Its Applications, Vol. \(32,\) No. \(2, 1998.\)
    
    \item\label{sohnikov2} A. Soshnikov, \textit{Universality at the edge of the spectrum in Wigner random matrices}, Comm. Math. Phys., Vol. \(207,\) No. \(3, 697-733, 1999.\)
    
    \item\label{soshnikov} A. Soshnikov, \textit{Poisson Statistics for the Largest Eigenvalue of Wigner Random Matrices with Heavy Tails}, Comm. in Probab., \(9, 82-91, 2004.\)
    
    \item\label{taovu} T. Tao, and V. Vu, \textit{Random matrices: Universality of local eigenvalue statistics}, Acta Math., Vol.  \(206,\) No. \(1, 127-204, 2011.\)
    
    \item\label{tracywidom} C. A. Tracy and H. Widom, \textit{Level spacing distributions and the Airy kernel},
    Comm. Math. Phys., Vol. \(159,\) No. \(1, 151-174, 1994.\)
    
    \item\label{wigner} E. P. Wigner, \textit{On the Distribution of the Roots of Certain Symmetric Matrices}, Annals of Mathematics, Second Series, Vol. \(67,\) No. \(2,\) \(325-327, 1958.\)

\end{enumerate}

\end{document}